\documentclass[reqno,envcountsame]{llncs}

\usepackage[T1]{fontenc}
\usepackage[latin1]{inputenc}
\usepackage{pgf}

\usepackage{amssymb}
\usepackage{amsmath}

\usepackage{amsthm}
\usepackage{fancyhdr}
\usepackage{setspace}
\usepackage{tikz}
\usepackage{here}
\usepackage{paralist} 
\usetikzlibrary{snakes}
\usetikzlibrary{arrows}

\usepackage{algorithm}

\usepackage{enumerate}
\def\ve#1{\mathchoice{\mbox{\boldmath$\displaystyle\bf#1$}}
{\mbox{\boldmath$\textstyle\bf#1$}}
{\mbox{\boldmath$\scriptstyle\bf#1$}}
{\mbox{\boldmath$\scriptscriptstyle\bf#1$}}}
\newcommand\vea{{\ve a}}
\newcommand\veb{{\ve b}}

\newcommand\vece{{\ve e}}

\newcommand\veg{{\ve g}}
\newcommand\veh{{\ve h}}

\newcommand\vep{{\ve p}}

\newcommand\ver{{\ve r}}

\newcommand\veu{{\ve u}}
\newcommand\vev{{\ve v}}
\newcommand\vew{{\ve w}}
\newcommand\vex{{\ve x}}
\newcommand\vey{{\ve y}}
\newcommand\vez{{\ve z}}
\newcommand\zero{{\ve 0}}

\newcommand\red[1]{\textcolor{red}{#1}}

\def\Circuits{{\mathcal C}}
\def\Cir{{\mathcal C}}
\def\Edges{{\mathcal E}}
\def\Edg{{\mathcal E}}
\def\Feas{{\mathcal F}}
\def\diamp{\Delta}
\def\Deltae{\Delta_{\Edg}}
\def\Deltac{\Delta_{\Cir}}
\def\diamfd{\Delta}

\DeclareMathOperator{\supp}{supp}

\newcommand{\N}{{\mathbb N}}

\newcommand{\R}{\mathbb R}

\newcommand\Csimple{$\mathcal{C}$-simple}
\newcommand\Csimplicity{$\mathcal{C}$-simplicity}

\newcommand\pmat[1]{\begin{pmatrix} #1 \end{pmatrix}}

\begin{document}

\title{On the Circuit Diameter Conjecture}
\author{Steffen Borgwardt\inst{1} \and Tamon Stephen\inst{2} \and Timothy Yusun\inst{2}}

\pagestyle{plain}

\institute{University of Colorado Denver \email{steffen.borgwardt@ucdenver.edu} \and Simon Fraser University \email{\{tamon,tyusun\}@sfu.ca}}

\maketitle

\begin{abstract}
From the point of view of optimization, a critical issue is relating the 
combinatorial diameter of a polyhedron to its number of facets $f$ and dimension $d$.
In the seminal paper of Klee and Walkup~\cite{kw-67}, the Hirsch conjecture of an
upper bound of $f-d$ was shown to be equivalent to several seemingly simpler statements, 
and was disproved for unbounded polyhedra through the construction of
a particular 4-dimensional polyhedron $U_4$ with 8 facets.  The Hirsch bound 
for bounded polyhedra was only recently disproved by Santos~\cite{s-11}.

We consider analogous properties for a variant of the combinatorial
diameter called the \emph{circuit} diameter. In this variant, the walks are built
from the circuit directions of the polyhedron, which are the minimal non-trivial
solutions to the system defining the polyhedron. 

We are able to prove that circuit variants of the so-called non-revisiting conjecture and $d$-step conjecture both imply the circuit analogue of the Hirsch conjecture. For the equivalences in
\cite{kw-67}, the wedge construction was a fundamental proof technique. We exhibit why it is not available in the circuit setting, and what are the implications of losing it as a tool. 

Further, 
we show the circuit analogue of the non-revisiting conjecture 
implies a linear bound on the circuit diameter of all unbounded polyhedra -- 
in contrast to what is known for the combinatorial diameter. 
Finally, we give two proofs of a circuit version of the $4$-step conjecture. 
These results offer some hope that the circuit version of the Hirsch conjecture 
may hold, even for unbounded polyhedra. 

A challenge in the circuit setting is that different realizations of polyhedra 
with the same combinatorial structure may have different diameters. 
We adapt the notion of simplicity to work with circuits in 
the form of {\em \Csimple} and {\em wedge-simple} polyhedra. We show that it 
suffices to consider such polyhedra for studying circuit analogues of the Hirsch
conjecture. 
\end{abstract}



\section{Introduction}


The \emph{combinatorial diameter} of a polyhedron is the maximum number of edges that is necessary to connect 
any pair of vertices by a walk.  It has been studied extensively due to its intimate connection to the 
simplex algorithm for linear programming: it is a lower bound on the best-case performance of the simplex algorithm, 
independent of the pivot rule used.  This motivates the question: {\em What is the largest possible 
combinatorial diameter of a $d$-dimensional convex polyhedron with a given number $f$ of facets?} 
In particular, if there is a family of polyhedra with combinatorial diameter that is exponential 
in $f$ and $d$, then a polynomial pivot rule cannot exist.

The quantity $f-d$, conjectured by Hirsch in the late 1950s~\cite{d-63} as a prospective
upper bound for the combinatorial diameter, remains a landmark in discussion of the combinatorial
diameter.  It is tight for the key special cases of $d$-cubes and $d$-simplices.
It is known to hold for important classes of polyhedra including $0/1$-polytopes \cite{n-89} and 
network flow polytopes \cite{bdf-16b}, but does not hold in general. For unbounded 
polyhedra, there is a $4$-dimensional counterexample \cite{kw-67} and for bounded polytopes a 
$20$-dimensional counterexample \cite{msw-15} (the original counterexample was of dimension $43$ \cite{s-11}). 
These counterexamples can be used to generate families of counterexamples in higher dimension, but it
is otherwise difficult to generate exceptions.
Indeed, the known counterexamples only give rise to a violation of the stated bound by 
a factor of at most $\frac{5}{4}$ for unbounded polyhedra and $\frac{21}{20}$ for bounded polytopes. 
It is open whether these bounds can be exceeded. See \cite{ks-10} for the state-of-the-art and \cite{bdhs-13} for some recent progress for low values of $f$ and $d$.



Recent avenues of research consider alternative models in an effort to understand 
\emph{why} the combinatorial diameter violates the Hirsch bound.  These include
working with combinatorial abstractions of polytopes, see e.g.~\cite{ehrr-10,lms-15}, 
and augmented pivoting procedures that are not limited to the vertices and edges of
the polytope, see e.g.~\cite{bfh-14,bdf-16,bdfm-16,dhl-15}.  We follow the latter path, focusing
on the model of \emph{circuit walks}, arguably the most natural of the augmented 
procedures.


\subsection{Circuit walks}

We assume that a polyhedron $P$ is given by an irredundant representation as full-dimensional polyhedron in $\R^d$. 
Let \[P=\{\vez\in \R^d: A \vez\geq \veb\}\]
for matrix  $A\in\R^{f\times d}$. The \emph{circuits} $\Circuits(A)$ of $A$ are the normalized vectors  $\veg$ for which $A\veg$ is support-minimal in  $\left\{A\vex: \ \vex \in \R^d \backslash \{0\}\right\}$. The normalization (with respect to any norm) gives a unique representative modulo multiplication with a positive constant. The circuits correspond to the so-called elementary vectors as introduced in \cite{r-69}. Note the set of circuits contains the edge directions of the polyhedron $P$. 

If needed, non-full-dimensional polyhedra can be represented in the form $P= \{\, \vez\in \R^{d'}:\ A^1\vez=\veb^1, A^2\vez\geq \veb^2 \,\}$  for matrices $A^i\in\R^{f_i\times d'}$ and vectors $\veb^i\in \R^{f_i}$, $i=1,2$.
In this form, the \emph{circuits} $\Circuits(A^1,A^2)$ of $A^1$ and $A^2$ are those (normalized) vectors $\veg \in \ker(A^1)\setminus\{\,\ve 0\,\}$, for which $A^2\veg$ is support-minimal in  $\left\{A^2\vex: \ \vex \in \ker\left(A^1\right)\backslash \{0\}\right\}$. 

We use $\mathcal{C}(P)$ to refer to the set of circuits of $P$ without explicit consideration of the underlying matrices.
Following Borgwardt et al.~\cite{bfh-14}, for two vertices $\vev^{(1)},\vev^{(2)}$ of $P$, we call a sequence $\vev^{(1)}=\vey^{(0)},\ldots,\vey^{(k)}=\vev^{(2)}$ a \emph{circuit walk of length $k$} if for all $i=0,\ldots,k-1$ we have
\begin{enumerate}
\item $\vey^{(i)}\in P$,
\item $\vey^{(i+1)}-\vey^{(i)}=\alpha_i\veg^{(i)}$ for some $\veg^{(i)}\in\Circuits(P)$  and $\alpha_i>0$, and
\item $\vey^{(i)}+\alpha\veg^{(i)}$ is infeasible for all $\alpha>\alpha_i$. 
\end{enumerate}  
Informally, a circuit walk of length $k$ takes $k$ steps of maximal length along circuits of $P$. The above properties are also well-defined when $\vev^{(1)}=\vey^{(0)}$ is not a vertex or when $\vev^{(2)}=\vey^{(k)}$ is not a vertex. In fact, we will sometimes use the terms in this more general sense. If this is the case, we will explicitly state that the walks at hand may start or end at a non-vertex. Note that properties $2.$ and $3.$ rule out the use of a circuit in the recession cone of $P$, which would give an unbounded step.

The \emph{circuit distance} from $\vev^{(1)}$ to $\vev^{(2)}$ then is the minimum number of steps of a circuit walk from $\vev^{(1)}$ to $\vev^{(2)}$. The \emph{circuit diameter} $\diamp_\Circuits(P)$ of $P$ is the maximum circuit distance between any two vertices of $P$. We denote the maximum circuit diameter that is realizable in the set of $d$-dimensional polyhedra with $f$ facets as $\diamfd_\Circuits(f,d)$. For the maximum combinatorial diameter in this class of polyhedra, we use $\diamfd_\Edges(f,d)$. Additionally, we use the terms $\diamp_\Feas(P)$, and $\diamfd_\Feas(f,d)$ for the corresponding notions for the so-called {\em feasible circuit walks} \cite{bdf-16}, where the walk does not have to take steps of maximal length, but is only required to stay feasible. To distinguish these walks from the original circuit walks, the original ones are sometimes called {\em maximal circuit walks}.

In \cite{bdf-16}, several concepts of walking along circuits with respect to different restrictions are compared, giving a hierarchy of circuit distances as well as diameters. We would particularly
like to understand the part of this hierarchy involving $\diamp_\Edges(P)$, $\diamp_\Circuits(P)$, and $\diamp_\Feas(P)$. Note that all edge walks are special (maximal) circuit walks; the restriction is to use only actual edges instead of any circuits for the directions of the steps. Further all maximal circuit walks are feasible walks; the restriction is to only use maximal step lengths. This gives the relation
 $$\diamfd_\Edges(f,d) \geq \diamfd_\Circuits(f,d) \geq \diamfd_\Feas(f,d).$$

Recall that the Hirsch conjecture (for the combinatorial diameter) is false, i.e. $\diamfd_\Edges(f,d)$ does violate the Hirsch bound $f-d$ \cite{kw-67,s-11}.
In contrast, it is possible to show $\diamfd_\Feas(f,d)\leq f-d$ for all $f,d$ \cite{bdf-16}.
This means that either between $\diamfd_\Edges(f,d)$ and $\diamfd_\Circuits(f,d)$, or between $\diamfd_\Circuits(f,d)$ and $\diamfd_\Feas(f,d)$ we lose validity of the bound $f-d$.
For the middle part of this hierarchy, i.e. the (maximal) circuit walks themselves, the corresponding claim is open for bounded and for unbounded polyhedra:

\conjecture[Circuit Diameter Conjecture -- Original formulation \cite{bfh-14}]\label{original}
{
 	\\For any $d$-dimensional polyhedron with $f$ facets the circuit diameter is bounded above by $f-d$.\\
}

Suppose Conjecture \ref{original} is true -- then there is a significant conceptual difference between walking along edges and walking along any circuits. If Conjecture~\ref{original} is not true, there is a significant
conceptual difference between taking steps of maximal length and just staying feasible. In fact, the distinction may be even finer; it could a priori occur that $\diamfd_\Edges(f,d) >
\diamfd_\Circuits(f,d) > \diamfd_\Feas(f,d)$. 

This is one of the many incentives to study Conjecture \ref{original} and leads to an interpretation as investigating why the Hirsch bound is violated for the combinatorial diameter. 

\subsection{Our contributions}

In their seminal paper \cite{kw-67}, Klee and Walkup gave several insights related
to the Hirsch bound for polyhedral diameter.  
They showed that it is enough to work 
with simple polyhedra, and showed that the general Hirsch conjecture is equivalent
to three more restricted statements.  These include the \emph{$d$-step conjecture},
which is the special case where $f = 2d$, and the \emph{non-revisiting conjecture},
that any two vertices can be joined by a walk that visits each facet at most once.
They then exhibited an unbounded 4-dimensional polyhedron $U_4$ with 8 facets that
has combinatorial diameter 5, disproving these results for unbounded polyhedra.
However, restricting attention to bounded polytopes, the equivalences again hold,
and here they show that, in contrast to the unbounded case, the $d$-step
conjecture \emph{does} hold at $d=4$ and $d=5$.  Polyhedra based on $U_4$ remained
the only known non-Hirsch polyhedra for more than four decades until
Santos~\cite{s-11} found counterexamples to the bounded versions of the conjectures, 
refuting $d$-step at $d=43$ and subsequently~\cite{msw-15} at $d=20$, the current record.


In this paper, we study the analogous questions in the context of circuit diameter.

One of the main challenges we face is that in the circuit context, different geometric realizations
of polyhedra with the same combinatorial structure may have different circuit diameters; 
see e.g. the example in \cite{bfh-14}. 
Thus we need to introduce a notion of simplicity which depends on the geometry of the problem,
including the circuits.  We call polyhedra which satisfy this property {\em \Csimple}, and show:
\begin{lemma}\label{csimple}
Let $P$ be a polyhedron. Then there is a \Csimple\ polyhedron $P'$ in the same dimension and with the same number of facets with $\diamp_\Circuits(P)\leq \diamp_\Circuits(P').$
\end{lemma}

In fact, we prove that it suffices to consider polyhedra $P$ with a slightly stronger property that also requires $P$ to remain \Csimple\ under repeated wedge operations. 
We call this \emph{$k$-wedge-simplicity} and show that Lemma~\ref{csimple} holds for $k$-wedge-simplicity as well.

The wedge construction is one of the most powerful tools in the studies of the combinatorial diameter. For example, it is used in the roundabout proof of equivalences of different variants of the Hirsch conjecture in \cite{kw-67} and in the construction of a counterexample to the Hirsch conjecture for bounded polytopes \cite{s-11}. 
We exhibit an example where circuit walks in a wedge do not project to circuit walks in the original polyhedron, which makes it difficult to transfer 
certain results to the circuit setting.

We nevertheless are able to recover generalizations of some of the results that hold 
in the combinatorial case. We show the equivalence of several variants of Conjecture \ref{original} 
for the circuit diameter: 

\begin{theorem}\label{maintheorem}
Consider the following statements:
\begin{enumerate}[(1)]
\item Let $\veu,\vev$ be two vertices of a \Csimple\ polyhedron $P$.
	Then there is a non-revisiting circuit walk from $\veu$ to $\vev$.
\item Let $\veu,\vev$ be two vertices of a \Csimple\ $d$-dimensional polyhedron $P$ with $2d$ facets.
	Then there is a non-revisiting circuit walk from $\veu$ to $\vev$.
\item $\diamfd_\Circuits(f,d)\leq f-d$ for all $f\geq d$
\item  $\diamfd_\Circuits(2d,d)\leq d$ for all $d$
\end{enumerate}
Then $(2) \iff (1) \Rightarrow (3) \iff (4)$.
\end{theorem}


The relationship of these conjectures is more involved than for the combinatorial diameter and require more technical detail. In particular, we have to carefully distinguish between statements for bounded and for unbounded polyhedra. To distinguish between bounded and unbounded quantities, we use superscripts: $\diamfd_\Circuits^u(f,d)$ denotes the maximal circuit diameter of an unbounded $d$-dimensional polyhedron with $f$ facets, while $\diamfd^b_\Circuits(f,d)$ denotes its counterpart for bounded $d$-dimensional polytopes. 

But we also find an interesting connection of the circuit diameters of unbounded polyhedra and bounded polytopes: We show that the existence of a non-revisiting circuit walk between any pair of vertices in every $\Circuits$-simple \emph{bounded} polytope would guarantee a linear bound on the diameter of all $\Circuits$-simple \emph{unbounded} polyhedra. We expect that $\Delta_\Circuits(f,d)$ is attained at an unbounded polytope, and thus $\Delta_\Circuits(f,d)= \Delta_\Circuits^u(f,d)$, but the proof for the combinatorial case does not easily carry over.


\begin{theorem}\label{thm:bounded_and_unbounded}
If all $\Circuits$-simple bounded $(f',d')$-polytopes with $f'\leq f+d-1$ and $d'\leq d$ satisfy the non-revisiting conjecture (Conjecture \ref{non revisiting}), then $\diamfd_\Circuits(f,d)\leq f-1$.
\end{theorem}

We would like to stress that Theorem \ref{thm:bounded_and_unbounded} gives a connection between the bounded and unbounded diameters of a type that is not known for the combinatorial diameter. If the non-revisiting conjecture can be resolved positively for $d$-dimensional bounded polytopes up to a certain number of facets (recall the combinatorial $d$-step conjecture holds for all \emph{bounded} polytopes of dimension $d\leq 6$ \cite{bs-11}), one immediately obtains a linear bound for unbounded polyhedra up to a slightly lower number of facets. Moreover, in case the non-revisiting conjecture is valid for circuit walks in all bounded polytopes, Theorem \ref{thm:bounded_and_unbounded} would give a general upper bound of the form
$$\diamfd_\Circuits(f,d) \leq  \diamfd_\Circuits^b(f,d) + d-1 \leq f-1.$$
In contrast, for the combinatorial diameter, one can provide a general quadratic  upper bound of the form\footnote{We thank the anonymous referee for the suggestion.}
$$ \Delta_\Edges(f,d) \leq (f-d)\Delta_\Edges^b(f+1,d)-2.$$
This can be seen as follows: Let $\bar P$ be an unbounded polyhedron with $f$ facets. Then add a facet $F$ ``at infinity'' to obtain a bounded polytope $P$ with $f+1$ facets. 

Now consider an edge walk in $P$ between vertices $\veu$ and $\vev$ (not in $F$).
If the walk uses an edge $\vece$ in $F$, replace the step by a shortest walk in the two-dimensional face $F'$ of $\bar P$ that contains $\vece$ -- this walk begins and ends at the neighbors of the vertices of $\vece$.
The Hirsch conjecture holds in dimension $2$, so the length of this path is at most $(f-(d-2)) -2=f-d$.
Replacing each such $e$ by the corresponding path in $F'$, and noting that the construction is only applied if we use two steps to enter and leave $F$, gives a bound of at most $(f-d)\Delta_\Edges^b(f+1,d)-2$ to connect $\veu$ and $\vev$ in $\bar P$. 

Recall that for the combinatorial diameter, a disproof of the Hirsch conjecture was much easier for the unbounded case, with a counterexample $U_4$ in dimension $4$ \cite{kw-67}.
We prove that $U_4$, in contrast to the combinatorial case, \emph{does} satisfy the Hirsch bound in the circuit setting, independent of realization.\footnote{A preliminary version of this result
appeared in the proceedings of Eurocomb 2015~\cite{sy-15b}.} 
Indeed, we show that for circuits, the 4-step conjecture holds even for unbounded polyhedra:

\begin{theorem}[Circuit 4-step]\label{thm:4step}
$\diamfd_\Circuits(8,4) = 4.$
\end{theorem}

We give two proofs of this fact.
The first uses the uniqueness of $U_4$, and the second strengthens the result of Santos et al.~\cite{sst-12} to get a $\veu$-$\vev$ walk in an arbitrary 4-spindle where a $\vev$-facet is entered at each step.


\vspace{2mm}
In summary, the results in this paper include:
\begin{itemize}
\item An adaption of the concept of simple polyhedra to the circuit context
 via \Csimple, wedge-simple and $k$-wedge-simple polyhedra. An example exhibiting the limitations of the wedge construction as a tool in the studies of circuit walks. {\em (Section \ref{sec:simple})}
\item Relations between various conjectures on the circuit
 diameter, including a circuit equivalent of the Hirsch conjecture and its $d$-step and 
 non-revisiting variants (Theorem~\ref{maintheorem}). Further, a proof that the circuit diameter conjecture for bounded polytopes implies a linear bound on the circuit
 diameter of unbounded polyhedra, provided the non-revisiting conjecture is true (Theorem~\ref{thm:bounded_and_unbounded}).
 {\em (Section \ref{sec:conjectures})}
\item Two proofs of the circuit 4-step conjecture (Theorem~\ref{thm:4step}) for unbounded polyhedra. The first one is based on showing that the circuit diameter conjecture holds in the most prominent 
 case where the combinatorial Hirsch conjecture fails (Theorem~\ref{U4thm}). The second proof extends a result of Santos et al.~\cite{sst-12}.  {\em (Section \ref{sec:4step})}
\item A brief discussion of related open questions. {\em (Section \ref{sec:discussion})}
\end{itemize}

%

\section{$\mathcal{C}$-Simplicity and Wedge-Simplicity}\label{sec:simple}
When considering bounds for the combinatorial diameter, we can restrict to simple polyhedra. This is because for any $d$-dimensional non-simple polyhedron with $f$ facets, a mild perturbation of the right-hand sides gives a simple polyhedron with the same number of facets and at least the same diameter \cite{ykk-84}. 

The key advantage of a simple polyhedron in the studies of the combinatorial diameter is that each step along an edge leaves exactly one facet and enters exactly one other facet. This makes an analysis significantly easier. 
To take advantage of non-edge circuits, walks will leave more than one facet at a time.
However, we can require that circuit steps do not arrive at more than one facet at a time.
The benefit of such walks has been observed in a special case before, see the last page of \cite{bfh-14}. 

In Section \ref{sec:Csimple}, we introduce such a property ($\mathcal{C}$-simplicity) and prove that $\diamfd_\Circuits(f,d)$ is realized by a \Csimple\
polyhedron. The wedge construction is one of the most powerful tools in the studies of combinatorial diameters. In Section \ref{sec:wedgesimple} we discuss a further specialization to polyhedra for which wedging over each of its facets still satisfies this property, and prove that $\diamfd_\Circuits(f,d)$ is always realized in this even more restrictive class of polyhedra. However, we also exhibit the limitations of wedging for studies of circuit walks. 

\subsection{$\mathcal{C}$-Simplicity}\label{sec:Csimple}


For $\vey^{(i)} \in P$, let $H^{(i)}$ denote the set of facets of $P$ that are incident to $\vey^{(i)}$. First, let us introduce some terminology for circuit walks, in which we enter only one new facet in each step.

\begin{definition}[Simple walks]
Let $P$ be a polyhedron. A circuit walk $\vey^{(0)},\ldots,\vey^{(k)}$ in $P$ is {\em simple} if $|H^{(i+1)}\backslash H^{(i)}|=1$ for $i = 1,2,\ldots,k$, where $H^{(i)}$ denotes the set of facets incident to $\vey^{(i)}$. Walks that violate this condition are called {\em non-simple}.
\end{definition}

We are particularly interested in polyhedra for which it suffices to only consider simple circuit walks. 
As the combinatorial diameter is an upper bound on
the circuit diameter of a polyhedron, for the study of Conjecture \ref{original}, it suffices to consider circuit walks of length at most $\diamfd_\Edges(f,d)$, which is bounded above for example by $f^{\log d +2}$ \cite{kk-92}. This leads to the following definition.

\begin{definition}[\Csimple]
Let $P$ be a polyhedron. We say $P$ is \Csimple\ if, for all pairs of vertices $\vev^{(1)}, \vev^{(2)}$, all shortest circuit walks from $\vev^{(1)}$ to $\vev^{(2)}$ are simple. 


Let $M$ be a finite set of points in $P$ that includes the set of vertices. If all shortest circuit walks starting at any point in $M$ and ending in any vertex $\vev$ are simple, then we say $P$ is {\em \Csimple\ with respect to $M$}.
\end{definition}

Note that $\mathcal{C}$-simplicity is a strictly stronger condition than simplicity of a polyhedron, as edge walks are a special type of circuit walks. Note further that it implies that only walks of length at most $\diamfd_\Edges(f,d)$ have to be considered for a polyhedron to be  \Csimple\, and only walks of length at most $\diamfd_\Edges(f,d) + d$ have to be considered for a polyhedron to be \Csimple\ with respect to $M$. Further, in our studies non-simple circuit walks will only appear in polyhedra that are not \Csimple. 

The goal for this section is to prove that it suffices to consider \Csimple\ polyhedra for the study of Conjecture \ref{original}, which leads to the following variant of the conjecture.

\conjecture[$\mathcal{C}$-Simplicity]\label{simple}
{
 	For any \Csimple\ $d$-dimensional polyhedron with $f$ facets the circuit diameter is bounded above by $f-d$.\\
}

We prove the equivalence of Conjecture \ref{simple} and Conjecture \ref{original} by showing that for fixed $f$ and $d$, $\diamfd_\Circuits(f,d)$ can be realized by a \Csimple\ polyhedron. We do so by
describing a perturbation of a polyhedron $P$ such that the perturbed polyhedron $P'$ is  \Csimple\ and has at least the same circuit diameter as $P$. 

The perturbations we consider are to the right hand sides of the defining inequalities, and thus do not change
the set of circuits, which depends only on $A$.  That is, the right-hand side $\veb$ is changed to $\veb \rightarrow \veb'=\veb + \vep$ for some vector $\vep$ with $\|\vep\|<\epsilon$ for a sufficiently small $\epsilon$. We call such a perturbation a \emph{mild perturbation}. The perturbed polyhedron is $P'=\{\vez\in \R^d: A\vez\geq \veb'\}$.  Note that for a mild perturbation, each facet remains a facet and the dimension does not change. 


The challenge here lies in the fact that the circuit diameter of a polyhedron depends on its realization, and not only its combinatorial structure (see \cite{bfh-14} for examples). Hence the effect of a perturbation, in theory, might reduce the diameter. We have to carefully check that there is a mild perturbation for which this is not the case.

{
\renewcommand{\thetheorem}{\ref{csimple}}
\begin{lemma}\label{Csimple}
Let $P$ be a polyhedron. Then there is a \Csimple\ polyhedron $P'$ in the same dimension and with the same number of facets with $\diamp_\Circuits(P)\leq \diamp_\Circuits(P').$
\end{lemma}
\addtocounter{theorem}{-1}
}

\begin{proof}
Let $P$ be a $d$-dimensional polyhedron with $f$ facets and let $P'$ denote a polyhedron derived by a mild perturbation. By definition of a mild perturbation, $P$ and $P'$ have the same number of $f$ facets and dimension $d$. Thus it suffices to prove that this perturbation can be performed such that $P'$ is \Csimple\ and has at least the same circuit diameter as $P$.

First, recall that $P$ and $P'$ share the same finite set of circuits. Further, observe that it suffices to only consider circuit walks of length at most $\diamfd_\Edges(f,d)$ to validate that $P'$ is
\Csimple. There are a finite number of points $\vey \in P'$ that may appear in such a walk. Hence the condition $|H^{(i+1)}\backslash H^{(i)}|=1$ only has to be satisfied for a finite set of pairs $(\vey^{(i)},\vey^{(i+1)})$. This implies that (for fixed $\epsilon$) the set of right-hand sides $\veb'=\veb + \vep$ with $\|\vep\|\leq \epsilon$ that do not give a \Csimple\ polyhedron $P'$ is of volume $0$. In turn, for any given $\epsilon$ there are infinitely many perturbations that yield a \Csimple\ polyhedron $P'$.

It remains to see that there is such a perturbation for which the circuit diameter of $P'$ is at least the circuit diameter of $P$. 
Let $Y_P$ denote the set of all points on circuit walks in $P$ of length at most $\diamfd_\Edges(f,d)$.
A simple but important observation is that 
the points in $Y_P$ are at least a certain fixed distance from each other. 
We say that a pair of points are \emph{close} if they are less than half the minimum distance between
pairs of points in $Y_P$.  Consider then a perturbation that is small enough that basic solutions in
$P'$ remain close to basic solutions in $P$.  Take $\vey \in P$ and $\vez \in P'$.
Let $I(\vey)$ denote the \emph{inner cone}, i.e.~the cone of feasible directions, of $\vey$ with respect 
to $P$ and $I(\vez)$ denote the inner cone of $\vez$ with respect to $P'$.  Then we have the following:

 
\begin{enumerate}
\item There is a one-to-many correspondence between vertices $\vey$ of $P$ and vertices $\vez$ of $P'$ where each $\vez$ is close to a unique $\vey$ and at least one $\vez$ is close to a given $\vey$
(and possibly many are). 
In particular, each vertex $\vez$ of $P'$ is associated to precisely one close vertex $\vey$ in $P$.
\item Let $\vey \in Y_P$, $\vez \in P'$ be close and let $\veg$ be a circuit in $I(\vey)\cap I(\vez)$. Then a step along $\veg$ from $\vez \in P'$ gives a $\vez' \in P'$ that is close to precisely one $\vey'\in Y_P$, which is derived from a step along $\veg$ from $\vey \in P$.
\item Let $\vey \in Y_P$, $\vez \in P'$ be close. Then a step along $\veg \in I(\vez) \backslash I(\vey)$ from $\vez$ will give a  $\vez' \in P'$ that is also close to precisely $\vey$.
\end{enumerate}
Let us consider a circuit walk $\vez^{(0)},\ldots,\vez^{(k')} \in P'$ for $k' \leq \diamfd_\Edges(f,d)$. Informally, the above properties tell us that it starts close to a vertex of $P$ (1.) and stays close to points in $Y_P$ in each step (2., 3.). More precisely, each $\vez^{(i)}$ is close to precisely one $\vey^{(i)} \in Y_P$. If (2.) is valid for $\vey^{(i)}, \vez^{(i)}$, then $\vey^{(i+1)}\neq \vey^{(i)}$. Else if (3.) holds, then $\vey^{(i+1)}= \vey^{(i)}$. This implies that each $\vez^{(0)},\ldots,\vez^{(k')}$ corresponds to a circuit walk $\vey^{(0)},\ldots,\vey^{(k)}$ in $P$ with $k\leq k'$. 

Let now $\diamp_\mathcal{C}(P)=k$ and let $\vey^{(0)},\ldots,\vey^{(k)}$ be a walk in $P$ realizing the diameter. Further, let $\vez^{(0)}$ be a vertex of $P'$ close to $\vey^{(0)}$ and $\vez^{(k')}$ be a vertex of $P'$ close to $\vey^{(k)}$. If the circuit distance of $\vez^{(0)}$ and $\vez^{(k')}$ is strictly less than $k$, then there is a circuit walk $\vez^{(0)},\ldots,\vez^{(k')}$ and $i,i'\leq k$ such that $\vez^{(i')}$ is close to $\vey^{(i)}$ for $i'<i$. By the above, we then know that the walk $\vez^{(0)},\ldots,\vez^{(i')}$ corresponds to a walk $\vey^{(0)},\vey'^{(1)},\ldots,\vey'^{(i')}=\vey^{(i)}$ of length $i'<i$. This implies that the circuit distance of $\vey^{(0)}$ and $\vey^{(k)}$ is less than $k$, a contradiction. Thus the circuit distance of $\vez^{(0)}$ and $\vez^{(k')}$ is at least $k$, which proves the claim.




\end{proof}

Hence we have proven the following:

\begin{corollary}\label{Csimple-reduce}
For any $f > d > 1$, $\diamfd_\Circuits(f,d)$ is attained by a \Csimple\ $d$-dimensional polyhedron with $f$ facets.
\end{corollary}

\subsection{Wedge-Simplicity and the Limitations of Wedges for Circuit Walks}\label{sec:wedgesimple}

One of the important tools for studying the combinatorial diameter is the well-known wedge construction. Let us recall a formal definition.

\begin{definition}[Wedge]\label{def:wedge}
Let $P$ be a $d$-dimensional polyhedron and let $F$ be a facet of $P$. A \emph{wedge} on $P$ over $F$  is a $(d+1)$-dimensional polyhedron $P'=H^\leq \cap (P \times L)$, where $P \times L$ denotes the product of $P$ with $L=[0,\infty)$ and $H^\leq \subset \R^{d+1}$ is a closed halfspace 
with $P\times\{0\} \subset H^\leq$ that is defined by a hyperplane $H$ that intersects the interior of $P\times L$ and satisfies 
$H \cap (P \times\{0\} )= F\times\{0\}$.
\end{definition}

Let $P=\{\vez\in \R^d: A \vez\geq \veb\}$ be a $d$-dimensional polyhedron and let $F$ be a facet of $P$ defined by $a^T\vez\geq b$. Then the $(d+1)$-dimensional polyhedron $$P'=\{(\vez^T,z_{d+1})^T\in \R^{d+1}: A \vez\geq \veb, 0\leq z_{d+1}\leq c\cdot(a^T\vez-b)\}$$ is a wedge for any fixed $c>0$.

Figure \ref{img:wedge} depicts an example. Note that we only consider the wedging operation when it is done over a facet of $P$. The operation can be extended to faces of smaller dimension, but
we do not use this here. Also, note that there is a distinction between `the' wedge on $P$ over $F$ (the combinatorial class) and `a' wedge on $P$ over $F$ (one realization in this combinatorial class; a different realization arises when choosing a different $H^\leq$). In what follows, which one we refer to will be clear from the context.


By construction, $P'$ has $f+1$ facets. The {\em lower base} $P_l= P\times \{0\}$ of $P'$  and the {\em upper base} $P_u=H\cap (P\times L)$ of $P'$ are facets of $P'$, and both are isomorphic to the original polyhedron $P$. The remaining $f-1$ facets of $P'$ are contained in spaces of the form $G\times L$, where $G\neq F$ is a facet of $P$; we call them the {\em sides} of the wedge. The lower base $P_l$ lies in the subspace $\mathbb{R}^d\times \{0\}$ while the upper base $P_u$ lies in the affine subspace $H$ of dimension $d$. 

We use $\phi$ to denote an isomorphism between  $\mathbb{R}^d\times \{0\}$ and $H$:  $\phi$ represents the projection of a vector from $\mathbb{R}^d\times \{0\}$ to $H$ corresponding to the product with $L$, and $\phi^{-1}$ represents the projection of a vector from $H$ to $\mathbb{R}^d\times \{0\}$ corresponding to the product with $L$. More precisely, for $H= \{\vez\in \R^d: z_{d+1}= c\cdot(a^T\vez-b)\}$, $\phi(\vez,0)=(\vez,c\cdot(a^T\vez-b))$ and $\phi^{-1}(\vez,c\cdot(a^T\vez-b))=(\vez,0)$.


\begin{figure}[htbp]
\begin{center}
\includegraphics[width=0.8\textwidth]{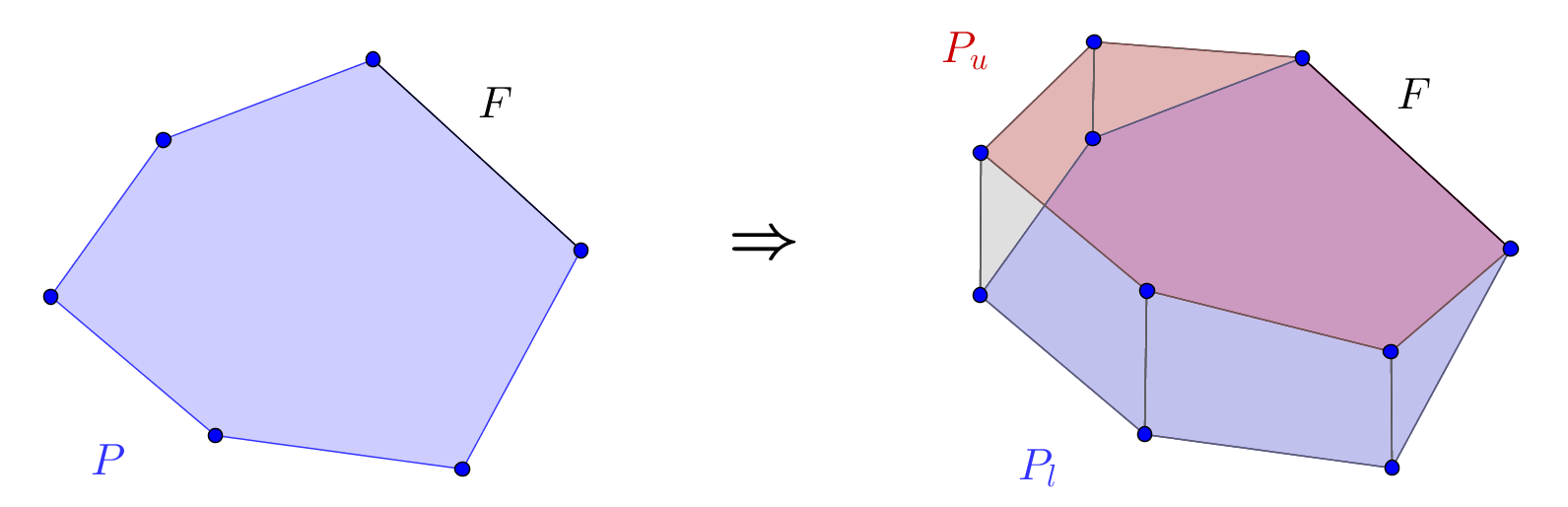}
\caption{The wedge $P'$ on the hexagon $P$ over facet $F$. Bases are $P_l$ and $P_u$.} \label{img:wedge}

\end{center}
\end{figure}

Let us take a look at the set of circuits of a wedge $\Circuits(P')$. The set $\Circuits(P')$ contains precisely the normalized vectors in the linear subspaces coming from the intersection of any subset of $d$ facets. The following lemma characterizes this set in terms of $\Circuits(P)$:

\begin{lemma}[Circuits of a wedge]\label{lemma:wedgecirc}
Let $P \subseteq \R^d \times \{0\}$ be a $d$-dimensional polyhedron with set of circuits $\Circuits(P)$ and let $F$ be one of its facets. Then if $P'$ is a wedge on $P$ over $F$, the set of circuits $\Circuits(P')$ is comprised of vectors of the form
\begin{enumerate}[(i)]
\item $(0,0,\ldots,0,\pm 1)^T \in \R^{d+1}$
\item $(\pm \veg,0)^T \in \R^{d+1}$, where $\veg \in \Circuits(P) \subseteq \R^d$
\item $\phi\left((\pm \veg,0)^T\right) \in \R^{d+1}$, where $\veg \in \Circuits(P) \subseteq \R^d$
\end{enumerate}
\end{lemma}
\begin{proof}
Each circuit direction of $P'$ is defined by a selection of $d$ facets with
linearly independent outer normals. 

(i) First, consider $d$ sides of $P'$ and recall that they correspond to $d$ facets $G_1,\dots,G_d$ of $P$.
If they intersect, their intersection $G_1 \cap \dots \cap G_d$ is a single point $\veu$ in $\mathbb{R}^d$ which may either be in $P$ (in which case it is a vertex) or not in $P$.
Hence the intersection of the corresponding sides of $P'$ is
\[ (G_1 \times L) \cap \cdots \cap (G_d \times L) = (G_1 \cap \cdots \cap G_d) \times L = \{\veu\} \times L,\]
which corresponds to the circuit direction $(0,0,\ldots,0,\pm 1)^T \in \R^{d+1}$.


(ii) Next, let the lower base $P_l$ be one of the $d$ facets in the intersection.
Then the other $d-1$ facets again correspond to facets $G_1,\dots,G_{d-1}$ of $P$: For the sides, these are the same facets as above; the facet corresponding to $P_u$ is $F$.
For this, the intersection of these $d$ facets corresponds to the intersection 
\[ (G_1 \times L) \cap \cdots \cap (G_{d-1} \times L) \cap P_l = G_1 \cap \cdots \cap G_{d-1},\]
which gives two circuits $\pm \veg\in \Circuits(P) \subseteq \mathbb{R}^d$.
We obtain circuits $(\pm \veg,0)^T\in \mathbb{R}^{d+1}$. 

(iii) Finally, let the upper base $P_u$ be one of the $d$ facets in the intersection.
Then the other $d-1$ facets correspond to facets $G_1,\dots,G_{d-1}$ of $P_u$; in particular, the facet corresponding to $P_l$ is $F$.
The facets $G_1,\dots,G_{d-1}$ are the projections of corresponding facets $G'_1,\dots,G'_{d-1}$ of $P_l$ with respect to $\phi$.
Due to this, we obtain circuits $\phi((\pm \veg,0)^T) \in \mathbb{R}^{d+1}$.

Note that $\phi((\pm \veg,0)^T)=(\pm \veg,0)^T$ for any $(\pm \veg,0)^T$ in the affine hull of $F$.
These circuits arise in both (ii) and (iii).
\end{proof}

Wedges are a basic building block for results on combinatorial diameters, due to some nice properties: 
\begin{enumerate}
\item Wedging over a facet $F$ of a simple polyhedron $P$ gives a simple polyhedron $P'$.
	This is because all vertices of $P'$ are contained in $P_l\cup P_u$, and both $P_l$ and $P_u$ are isomorphic to the simple polyhedron $P$.
\item The wedge satisfies $\diamp_\Edges(P') \leq \diamp_\Edges(P) + 1$, as for any vertex $\vev\in P_u$ there is a neighboring vertex $\veu\in P_l$; in fact $\vev=\phi(\veu)$.
\item Any edge walk in $P'$ transfers to an edge walk in $P$ by projecting to $\R^d\times\{0\}$ using $\phi^{-1}$.
	Because of this, we obtain $\diamp_\Edges(P)\leq \diamp_\Edges(P')$.
\end{enumerate}

For circuit walks in wedges, the situation is more involved.
Clearly all wedges $P'$ on $P$ over a facet $F$ satisfy $\diamp_\Circuits(P') \leq \diamp_\Circuits(P) + 1$, due to the neighboring vertices $\vev\in P_u$ and $\veu\in P_l$ (see $2.$ above).
In contrast, circuit walks in $P'$ do not necessarily transfer to circuit walks in $P$.
This is because it is possible to hit the interior of $P_u$, as depicted in Figure \ref{img:wedge_circuit}: Let a walk begin at a vertex $\veu\in P_l\backslash F$.
Then take a step along a circuit in $\Circuits(P_u)$, which will give a $\vey$ in a side of the wedge.
Then continue from $\vey$ along a circuit in $\Circuits(P_l)$ to $\vey'$, which may lie in the interior of $P_u$.
Projecting this walk down to $P_l$ gives a circuit step in $P$ that does not use maximal step length and is thus not a circuit walk.
The point $\vey'$ may also be incident to both $P_u$ and a side of the wedge, or lead to another point where this happens, in which case $P'$ is not \Csimple\ -- and this is possible even if $P$ itself is \Csimple.
See Figure \ref{img:wedge_notCsimple}.

\begin{figure}[htbp]
\begin{center}
	\begin{minipage}{0.3\textwidth}
	\centering
	\includegraphics[width=\textwidth]{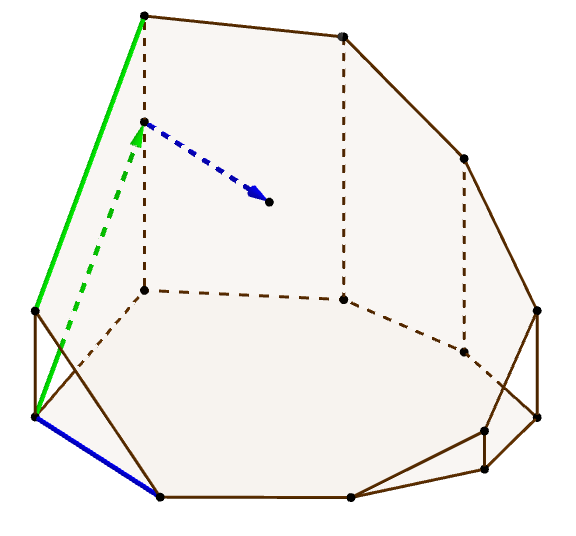}
	\end{minipage}
	\qquad
	\begin{minipage}{0.3\textwidth}
	\centering
	\includegraphics[width=\textwidth]{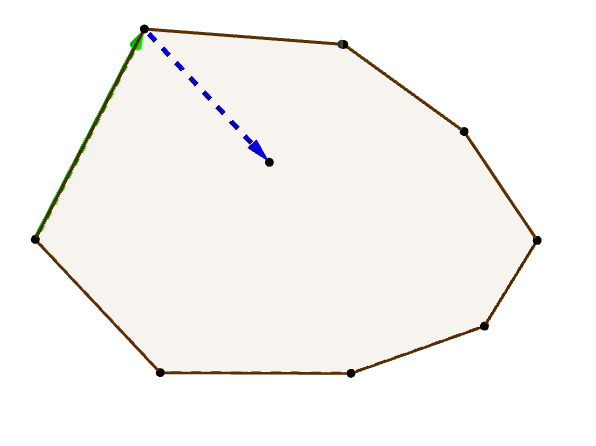}
	\end{minipage}
\caption{A circuit walk in $P'$ that does not project to a circuit walk in $P_l.$} \label{img:wedge_circuit}
\end{center}
\end{figure}

Thus the corresponding circuit formulations of the above Properties $1$ and $3$ do not hold.
In particular, the wedge operation may reduce the circuit diameter by creating `shortcuts' between vertices of $P_l$ by walks that take an intermediary step into the interior of $P_u$.
In fact, it is possible to construct a polyhedron $P$ and a wedge $P'$ with $\diamp_\Circuits(P)> \diamp_\Circuits(P')$ (based on the construction for Lemma $15$ in \cite{bdf-16}).

While the limitations of circuit walks with respect to Property $3$ are fundamental, a satisfactory solution can be found to address Property $1$.
On one hand, a wedge operation may create a $P'$ that is not \Csimple, even if the underlying polyhedron $P$ is \Csimple.
On the other hand, we are able to prove that $\Delta_\Circuits(f,d)$ is realized in a class of polyhedra for which this is not the case.
First, we require some new terminology that strengthens the \Csimple\ property.

\begin{figure}[htbp]
\begin{center}
\includegraphics[width=0.35\textwidth]{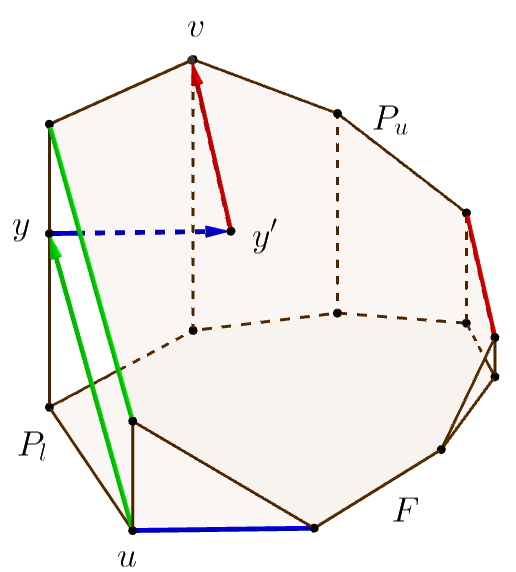}
\caption{Polyhedron $P=P_l$ is \Csimple\ but not wedge-simple: the walk $\veu,\vey,\vey',\vev$ is a non-simple circuit walk in $P'$ -- two new facets are hit in the step from $\vey'$ to $\vev$.} \label{img:wedge_notCsimple}
\end{center}
\end{figure}

\begin{definition}[Wedge-simple]
Let $P$ be a \Csimple\ polyhedron.
$P$ is a {\em [1-]wedge-simple polyhedron with respect to a facet $F$} if a wedge $P'$ on $P$ over $F$ is \Csimple.
$P$ is {\em wedge-simple} if for all facets $F$ a wedge $P'$ on $P$ over $F$ is \Csimple.
We can then recursively define $P$ to be {\em $k$-wedge-simple} for $k \ge 2$, if for all facets $F$ a wedge $P'$ on $P$ over $F$ is $(k-1)$-wedge-simple.
\end{definition}

Note that all wedges over the same facet are affinely isomorphic.
Thus if there is a wedge over a facet that is \Csimple, then so are all wedges over that facet.
This allows the simple wording ``a wedge is \Csimple'' (instead of ``there is a wedge that is \Csimple'') in the definition.
With this terminology, we are ready to prove that a \Csimple\ polyhedron can be perturbed to obtain a wedge-simple polyhedron.

\begin{lemma}\label{wedgesimple}
Let $P$ be a \Csimple\ polyhedron.
Then there is a wedge-simple polyhedron $P^*$ in the same dimension and with the same number of facets with $\diamp_\Circuits(P)\leq \diamp_\Circuits(P^*).$
\end{lemma}

\begin{proof}
Let $P'$ be a wedge on $P$ over facet $F$, and consider a mild perturbation of the facets of $P'$.
Perturbing a side of $P'$ has the same effect as perturbing the corresponding facet in $P$ before wedging; perturbing either base $P_l$ or $P_u$ has the same effect as perturbing the facet $F$ of $P$.
Thus, it is possible to guarantee some properties of $P'$ by using a mild perturbation of $P$ before the application of the wedge operation.

First, recall Lemma \ref{Csimple} telling us that any polyhedron can be perturbed to get a \Csimple\ polyhedron with at least the same circuit diameter.
Thus if $P'$ is not \Csimple, it can be perturbed mildly, which gives a polyhedron that is \Csimple\ and of at least the same circuit diameter.
But then $P'$ is just a wedge over a mildly perturbed $P$, by the above arguments.
Thus it is always possible to perturb an original polyhedron such that the wedge over a given facet $F$ is \Csimple.

So let now $P$ allow a wedge $P'$ over facet $F$ that is \Csimple.
Further, let $F'\neq F$ be a different facet of $P$ and consider a wedge $P''$ over facet $F'$.
If $P''$ is not \Csimple, again a mild perturbation of $P''$ will make it \Csimple.
This perturbation is also a perturbation of $P$, so that it is necessary to consider what happens with the wedge over facet $F$.

Note that a sufficiently small mild perturbation will maintain \Csimplicity\ using similar arguments as in the proof of Lemma \ref{Csimple}:
we perturb by less than the minimum distance between points in $Y_{P'}$.
But now a perturbation of $P$ is also a perturbation of the wedge $P'$ over $F$, and thus $P'$ stays \Csimple\ for this perturbation.

This means that one can iteratively apply perturbations to the original polyhedron $P$ to guarantee that the wedges over all facets $F$ are \Csimple. With each of these perturbations, all facets for
which there was a \Csimple\ wedge keep this property. There are only a finite number of facets of a given polyhedron. This proves the claim.
\end{proof}

Additionally, Lemma \ref{wedgesimple} transfers to $k$-wedge-simplicity.

\begin{corollary}\label{wedgesimple2}
Let $P$ be a \Csimple\ polyhedron and let $k\in \N$ be given. Then there is a $k$-wedge-simple polyhedron $P^*$ in the same dimension and with the same number of facets with $\diamp_\Circuits(P)\leq \diamp_\Circuits(P^*).$
\end{corollary}

\begin{proof}
Let $P=P_0,P_1,\dots,P_k$ be a sequence of wedges $P_{i+1}$ on $P_i$ for $0\leq i<k$ and suppose w.l.o.g. $P_k$ is not \Csimple. This can be amended by a slight perturbation of $P_k$ (Lemma \ref{Csimple}), which translates to a perturbation of $P_{k-1}$, which in turn translates to a perturbation of $P_{k-2}$ and so on up to a perturbation of $P=P_0$. Repeated application of the proof of Lemma \ref{wedgesimple} thus gives the claim. 
\end{proof}

The above statements sum up to the equivalence of the following conjecture to the original formulation in Conjecture \ref{original}.

\conjecture[Wedge-Simplicity]\label{conj:wedgesimple}
{
 	For any $k$-wedge-simple $d$-dimensional polyhedron with $f$ facets and any $k\in \N$ the circuit diameter is bounded above by $f-d$.\\
}

In light of the many applications of wedging over polyhedra in the studies of the combinatorial diameter, Conjecture \ref{conj:wedgesimple} may be useful in its own right. It tells us that we may restrict our studies to \Csimple\ polyhedra for which an arbitrarily large number of wedging operations still gives a \Csimple\ polyhedron.

We conclude this section with two lemmas showing how $k$-wedge-simplicity transfers from a polyhedron $P$ to its faces.
Proofs of these are included in the Ph.D.~thesis of the third author~\cite{y-17}.

\begin{lemma}\label{lemma:wedge_transfer}
Let $P$ be a wedge-simple $d$-dimensional polyhedron and $F$ be any $d'$-face of $P$ for $1 < d' < d$. Then $F$ is also wedge-simple.
\end{lemma}

\begin{lemma}\label{lemma:kwedge_transfer}
Let $P$ be a $k$-wedge-simple $d$-dimensional polyhedron and $F$ any $d'$-face of $P$ for $1 < d' < d$. Then $F$ is also $k$-wedge-simple.
\end{lemma}
%
%

\section{The Conjectures}\label{sec:conjectures}

Next, we prove the equivalence of several variants of Conjecture \ref{original}. Most of them are circuit analogues of variants of the Hirsch conjecture for the combinatorial diameter. We begin with a discussion and formal statements for these conjectures before turning to the proof.

\subsection{Non-revisiting circuit walks}\label{sec:nonrevisiting}

One of the most useful variants in the studies of the Hirsch conjecture for the combinatorial diameter is the {\em non-revisiting conjecture}.
A walk is \emph{non-revisiting} if no facet is left during the walk then entered again at a later step; the non-revisiting conjecture was that any two vertices are connected by such a walk.
In particular this means that if two vertices $\veu,\vev$ lie in the same face of a polyhedron, there would be a walk connecting the two vertices that stays in this face and adheres to the given bound.

In an edge walk in a simple polyhedron, at each step exactly one facet is left and exactly one other facet is entered.
The non-revisiting conjecture requires that the entered facet is new, i.e. it has not been left before.
For circuit walks the situation is a bit more complicated.
Circuit steps that are not along edges will leave multiple facets simultaneously, and, in a non-\Csimple\ polyhedron, may enter multiple facets as well.
To transfer the concept of a non-revisiting walk, we consider the facets that are entered during a walk.

Recall that the connection of the non-revisiting conjecture for edge walks and the original Hirsch conjecture, in fact, comes from the `positive' interpretation of the above: one enters exactly one new facet in each step.
As one starts at a vertex, to which $d$ facets are incident, entering a new facet in each step immediately gives the bound $f-d$.
We use this interpretation to give a viable formulation for circuit walks.

\conjecture[Non-revisiting]\label{non revisiting}
{
For any polyhedron $P$ and two vertices $\veu,\vev\in P$, there is a circuit walk from $\veu$ to $\vev$ that enters a new facet in each step, that is, each step produces an active facet that was inactive at all previous steps. \\
}

Each circuit step enters at least one new facet, and may leave any number of old facets.
So generally, it is possible to enter `old' facets in a step as long as one also enters a new facet.
However, for \Csimple\ polyhedra, only exactly one new facet is entered in each step.
Then the above formulation is equivalent to asking for a circuit walk from $\veu$ to $\vev$ that does not enter a facet it left before.


It is easy to see that Conjecture \ref{non revisiting} implies Conjecture \ref{original} with the same arguments as before: one begins at a vertex, to which (at least) $d$ facets are incident, and enters a new facet in each step, which gives a bound of $f-d$.

It is also of interest to state this conjecture for the particular case $f = 2d$:
\conjecture[Non-revisiting $f = 2d$]\label{non revisiting d}
{
For any $d$-polyhedron $P$ with $2d$ facets, and two vertices $\veu,\vev\in P$, there is a circuit walk from $\veu$ to $\vev$ that enters a new facet in each step. \\ 
}

The motivation to explicitly state this special case of Conjecture~\ref{non revisiting} is that we will later find both conjectures to be equivalent.

\subsection{Any start}

In contrast to edge walks, which only walk between vertices of a polyhedron, one may consider circuit walks that begin at a feasible point in a polyhedron that is not a vertex: indeed many steps in circuit walks do not begin at vertices.
To study partial sequences of circuit walks (and for many other reasons), we here present a variant of the conjectures that deals with the necessary generalization of the starting point.
We say that a set of facets is linearly independent if the corresponding outer normals are linearly independent.

\conjecture[Any start]\label{any start}
{
 For any $d$-dimensional polyhedron $P$ with $f$ facets and any finite set $M$ of points in $P$ that includes the set of vertices, there is a circuit walk length from any point $\veu\in M$ to any vertex in $\vev$ that enters a new facet in each step.\\
}

Note the number of active facets $d'$ for any vertex $\veu$ of a $d$-dimensional polyhedron is at least $d$.
For all other points in the polyhedron, $d'<d$.
The above conjecture gives rise to a generalized notion of circuit walk that may start anywhere in the polyhedron, not only at a vertex.
The number of steps of such a circuit walk is bounded by $f-d'$.

Let us briefly turn to an example: Consider a simplex in $\R^d$ and recall it has $d+1$ vertices and $d+1$ facets, which implies $f-d=1$.
It has combinatorial diameter $1$, which transfers to circuit diameter $1$.
However, the number of steps to a vertex can be much larger for a start at a non-vertex.
For example, let a walk begin at a feasible point $\veu$ in the strict interior of a facet $F$.
Then $d'=1$ and $f-d'=d$.
Walking to the unique vertex that is not incident to $F$ requires exactly $d$ steps (if we assume the simplex to be \Csimple\ with respect to a set $M$ of points that includes the starting point).
Figure \ref{fig:simplex} depicts an example in dimension $3$.

\begin{figure}[htbp]
\begin{center}
\includegraphics[width=0.3\textwidth]{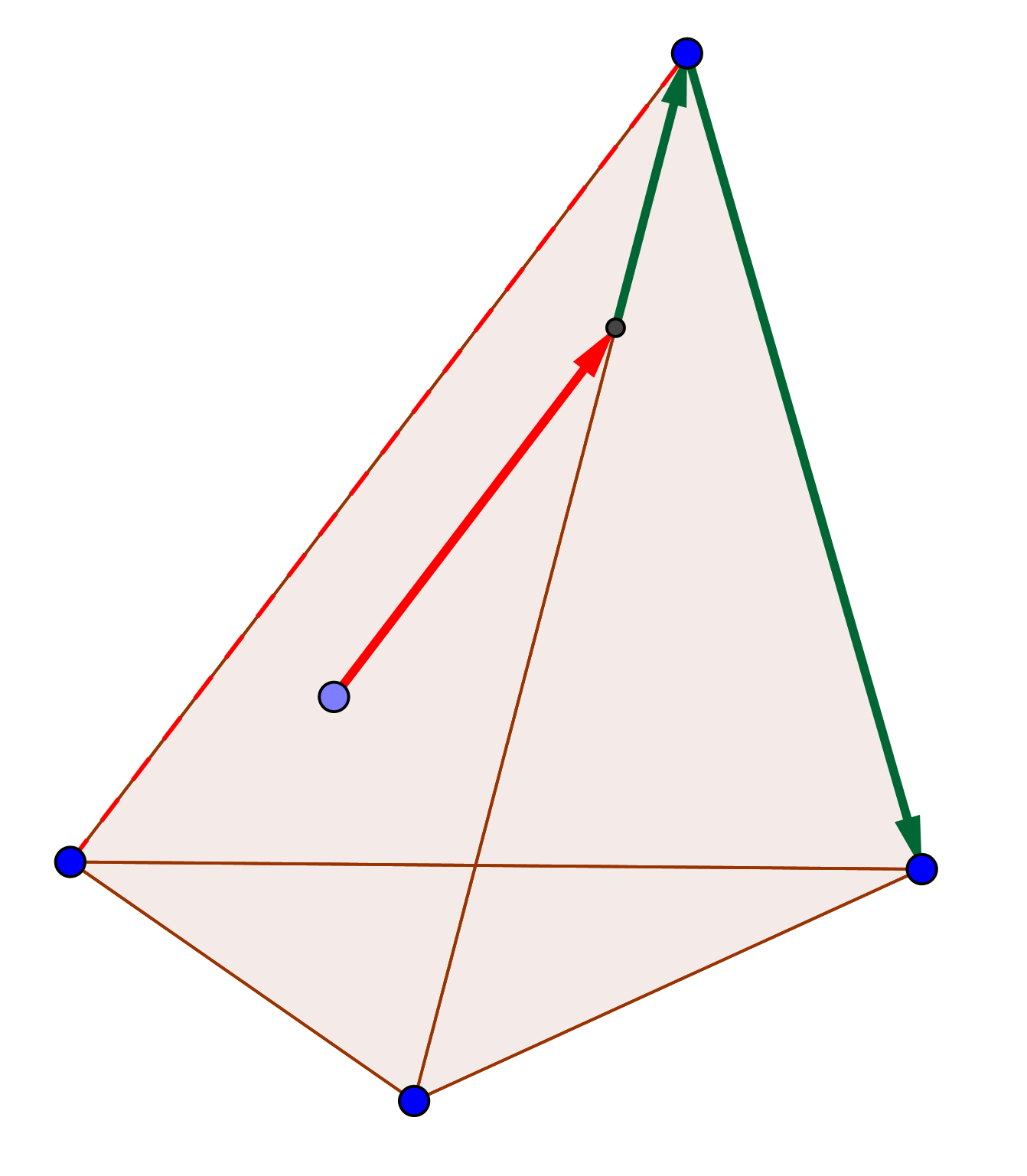}
\caption{A simplex in $\R^3$. The circuit distance of a point on the boundary to any vertex is at most three.} \label{fig:simplex}
\end{center}
\end{figure}

%
%
%

Consider a non-revisiting walk (as in Conjecture~\ref{non revisiting}) starting at a non-vertex.
Picking up a new facet in each step that did not appear before would transfer to a bound of $f-d'$, as only $d'\leq d$ facets are active in the beginning.
It is clear that Conjecture~\ref{any start} is at least as strong as Conjecture~\ref{original}, as it encompasses the corresponding statement for $M$ as the set of all vertices and due to $d'\geq d$ for all vertices.
But Conjecture~\ref{non revisiting} also implies Conjecture~\ref{any start}:

\begin{lemma}\label{lem:nonrev}
Let $P$ be a $d$-dimensional polyhedron with $f$ facets, let $\veu$ be a feasible point in $P$ that is incident to $d'$ linearly independent facets of $P$, and let $\vev$ be a vertex of $P$.
Suppose further that the non-revisiting conjecture (Conjecture \ref{non revisiting}) is true.
Then there is a circuit walk from $\veu$ to $\vev$ that enters a new facet in each step. 
\end{lemma}

\begin{proof}
Our strategy for the proof is as follows: We construct a polyhedron $P'\subset P$ such that $\veu,\vev$ are vertices of $P'$.
Then we show that a non-revisiting walk from $\veu$ to $\vev$ in $P'$ transfers to a non-revisiting circuit walk in $P$.

Let $F^v_1,\dots, F^v_{d}$ be $d$ linearly independent facets incident to $\vev$ in $P$ with outer normals $\vea^v_1,\dots,\vea^v_{d}$.
Let further $F^u_1,\dots,F^u_{d'}$ be $d'$ linearly independent facets incident to $\veu$ in $P$ and let $\vea^u_1,\dots,\vea^u_{d'}$ be the corresponding outer normals.
For $d'\geq d$, there is nothing to prove.
For $d' < d$, there are $d-d'$ outer normals among $\vea^v_1,\dots,\vea^v_{d}$, without loss of generality $\vea^v_1,\dots,\vea^v_{d-d'}$, such that the set $\{\vea^u_1,\dots,\vea^u_{d'}, \vea^v_1,\dots,\vea^v_{d-d'}\}$ is linearly independent.

Let now $F^{\geq}_{d'+i}=\{\vex\in \mathbb{R}^d: (\vea^v_i)^T\vex\geq (\vea^v_i)^T\veu\}$ for $i\leq d-d'$.
Note that by definition of the $\vea_i$, $(\vea^v_i)^T\vev\geq (\vea^v_i)^T \veu$ for all $i\leq d-d'$.
Thus $P'=P\cap \bigcap\limits_{i=1}^{d-d'} F^\geq_i$ contains both $\veu$ and $\vev$, and both are vertices of $P'$. 
Informally $P'$ is the intersection of $P$ with a cone (starting at $\veu$) of facets that are parallel to facets incident to $\vev$.

By validity of Conjecture \ref{non revisiting} there is a walk from vertex $\veu$ to vertex $\vev$ in $P'$ that is non-revisiting.
As $P$ and $P'$ only differ in facets incident to the starting point $\veu$ of the walk, this implies that in each step of such a walk there is a facet from the original polyhedron $P$ that bounds the step length.
Thus the corresponding walk is also a circuit walk in $P$ and it is non-revisiting in $P$.

\end{proof}

We obtain the following equivalence.

\begin{corollary}
The `any start' conjecture (Conjecture \ref{any start}) is true if and only if the non-revisiting conjecture (Conjecture \ref{non revisiting}) is true.
\end{corollary}


\subsection{Dantzig figures and the circuit $d$-step conjecture}

Next, we consider the connection of Conjecture \ref{original} to the so-called $d$-step conjecture and Dantzig figures.
It is well known that for the maximal combinatorial diameter of $d$-dimensional polyhedra, it suffices to consider polyhedra with $2d$ facets.
The maximal value of $f-d$ is realized in a polyhedron with $2d$ facets. This leads to the circuit equivalent of the $d$-step conjecture.

\conjecture[$d$-step]\label{d-step}
{
 For any $d$-dimensional polyhedron with $2d$ facets the circuit diameter is bounded above by $d$.\\
}

With $f=2d$, clearly $f-d=d$.
Conjecture \ref{d-step} treats a special class of polyhedra, and thus is a specialization of Conjecture \ref{original}.
But just as for the combinatorial diameter, we will see that both conjectures are equivalent.

For the combinatorial diameter, one may even restrict the studies to the so-called {\em Dantzig figures}.
First, let us define these figures. We follow \cite{ykk-84}.

\begin{definition}[Dantzig figure]\label{def:dantzig}
Let $P$ be a $d$-dimensional polyhedron with $2d$ facets, of which exactly $d$ are incident to a vertex $\veu$ and the other $d$ are incident to a vertex $\vev$.
Then the tuple $(P,\veu,\vev)$ is a {\em Dantzig figure}.
\end{definition}

Essentially, Dantzig figures are the intersection of two $d$-dimensional cones of $d$ facets.
For the combinatorial diameter, it suffices to consider the distances of $\veu$ and $\vev$ in such a Dantzig figure to resolve the Hirsch conjecture.
We state the circuit analogue.

\conjecture[Dantzig figure]\label{dantzig}
{
 For any $d$-dimensional Dantzig figure and vertices $\veu,\vev$ not sharing a facet, the circuit distance from $\veu$ to $\vev$ is bounded above by $d$.\\
}

Note that for a \Csimple\ Dantzig figure, the circuit distance from $\veu$ to $\vev$ is at least $d$.
If it is equal to $d$, then a corresponding walk is non-revisiting. 

A proof of the equivalence of the original Hirsch conjecture and the Dantzig figure variant for the combinatorial diameter relies heavily on the wedge construction.
More precisely, it makes fundamental use of the fact that edge walks in a wedge project down to edge walks in the original polyhedron.
As we saw in Section $2$, this is not true for circuit walks.
We turn to this issue in more detail after Theorem \ref{maintheorem} in the next section.


\subsection{On the relation of the conjectures}\label{sec:equivalences}

We now prove a sequence of implications relating Conjectures \ref{original}, \ref{non revisiting}, \ref{non revisiting d}, and \ref{d-step} (and their respective bounded versions).
The proof methods are inspired by the corresponding proof for the combinatorial diameter in \cite{kw-67} and \cite{ykk-84}, but we have to pay significantly more attention to technical detail.

First, we prove two lemmas about the circuit diameter of a product of two polyhedra.

\begin{lemma}\label{product}
	Let $P = \{\vex \in \R^{d_1} : A\vex \geq \veb\}$ and $Q = \{\vex \in \R^{d_2} : C\vex \geq \vew\}$ be pointed polyhedra (i.e.~they have at least one vertex).
Then $\Cir(P \times Q) = (\Cir(P) \times \{\zero\}) \cup (\{\zero\} \times \Cir(Q))$.
\end{lemma}
\begin{proof}
	By definition, $\Cir(P)$ consists of vectors $\veg \ne \zero$ for which $(A\veg)$ is support-minimal in $K_P = \{A\veg : \veg \in \R^{d_1} \ \backslash \ \{\zero\}\}$.
	Similarly, $\Cir(Q)$ are those vectors $\veh \ne \zero$ such that $(C\veh)$ is support-minimal in $K_Q = \{C\veh : \veh \in \R^{d_2} \ \backslash \ \{\zero\}\}$. We would like to characterize the set of circuits $\Cir(P \times Q)$.

Recall that \[ P \times Q = \left\{(\vex,\vey) \in \R^{d_1 + d_2} : A\vex \geq \veb, C\vey \geq \vew \right\}. \]

\noindent Let $M=\pmat{A & \zero \cr \zero & C}$.
Then $\Cir(P \times Q)$ is composed of vectors $(\veg,\veh) \in \R^{d_1 + d_2}$ such that \[\pmat{A & \zero \cr \zero & C}\pmat{\veg \cr \veh} = M\pmat{\veg \cr \veh} = \pmat{A\veg \cr C\veh}\] is support-minimal in the set \[ K_{PQ} = \left\{M\pmat{\veg \cr \veh} : \veg \in \R^{d_1} \ \backslash \ \{\zero\},\veh \in \R^{d_2} \ \backslash \{\zero\} \right\}   .\]

\noindent First, we claim that \[\veg \in \Cir(P) \iff \pmat{\veg \cr \zero} \in \Cir(P \times Q).\]
This follows from the fact that $A\veg$ is support-minimal in $K_P$ if and only if $M \pmat{\veg \cr \zero} = \pmat{A\veg \cr \zero}$ is support-minimal in $K_{PQ}$. Note that the assumption of
pointedness of $P$ and $Q$ is important here, since it implies that $P$ and $Q$ only have a trivial lineality space. This means that there is no nonzero vector $\veg \in \R^{d_1}$ such that $A\veg = \zero$, that is, $A\veg$ has nonempty support as long as $\veg \ne \zero$.
Similarly, $C\veh \ne \zero$ for nonzero $\veh \in \R^{d_2}$.

Hence, given a fixed nonzero $\veg \in \R^{d_1}$ for which $A\veg$ is support-minimal, $\supp \pmat{A\veg \cr \zero}$ cannot contain $ \supp\pmat{A\veg \cr C\veh}$. So, $M\pmat{\veg\cr\zero}$ is support-minimal in $K_{PQ}$, implying $\pmat{\veg\cr\zero} \in \Cir(P \times Q)$. Thus $\Cir(P) \times \{\zero\} \subseteq \Cir(P \times Q)$.


The argument for $\{\zero\} \times \Cir(Q) \subseteq \Cir(P \times Q)$ is similar.
To show that $\Cir(P \times Q)$ has no other circuits, we only need to prove that if $\pmat{\veg\cr\veh} \in \R^{d_1 + d_2}$ where both $\veg$ and $\veh$ are nonzero, then $\pmat{\veg\cr\veh}$ cannot be a circuit of $P \times Q$.
But this is immediate, since \[\supp \left(M\pmat{\veg\cr\zero}\right) \subsetneq \supp\left(M\pmat{\veg\cr\veh}\right)\] again because pointedness implies $C\veh \ne \zero$ for nonzero $\veh$.
The result follows.
\end{proof}

We obtain an immediate consequence on the circuit diameter of the product $P \times Q$.

\begin{lemma}\label{circdiamprod}
$\Deltac(P \times Q) = \Deltac(P) + \Deltac(Q)$.
\end{lemma}
\begin{proof}
Lemma~\ref{product} implies that when taking a step along a circuit direction in the product $P \times Q$, the position in one of the factors remains the same.
Hence, any circuit walk in $P \times Q$ can be decomposed into two circuit walks, one in $P$ and one in $Q$, giving the bound.
\end{proof}

Now, we are ready to discuss the relation of the conjectures.

{
\renewcommand{\thetheorem}{\ref{maintheorem}}
\begin{theorem}
Consider the following statements:
\begin{enumerate}[(1)]
\item Let $\veu,\vev$ be two vertices of a \Csimple\ polyhedron $P$.
	Then there is a non-revisiting circuit walk from $\veu$ to $\vev$.
\item Let $\veu,\vev$ be two vertices of a \Csimple\ $d$-dimensional polyhedron $P$ with $2d$ facets.
	Then there is a non-revisiting circuit walk from $\veu$ to $\vev$.
\item $\diamfd_\Circuits(f,d)\leq f-d$ for all $f\geq d$.
\item  $\diamfd_\Circuits(2d,d)\leq d$ for all $d$.
\end{enumerate}

Then $(2) \iff (1) \Rightarrow (3) \iff (4)$.
\end{theorem}
\addtocounter{theorem}{-1}
}

\begin{proof}
We prove this sequence of implications by showing first $(2) \Leftarrow (1) \Rightarrow (3) \Rightarrow (4)$.
Then we show $(2) \Rightarrow (1)$ and $(3) \Leftarrow (4)$.

$(2) \! \Leftarrow (1)$: 
If the non-revisiting conjecture holds for general $(f,d)$, then it holds for the particular choice of $f = 2d$.

$(1) \Rightarrow (3)$:
First, recall that $\diamfd_\Circuits(f,d)$ is realized by a \Csimple\ polyhedron $P$ by Corollary~\ref{Csimple-reduce}.
In a non-revisiting circuit walk from vertex $\veu$ to $\vev$, a new facet is entered in each step.
As $P$ is \Csimple, this is exactly one new facet per step.
The vertex $\veu$ is incident to exactly $d$ facets, thus there are at most $f-d$ steps.

$(3) \Rightarrow (4)$: $\diamfd_\Circuits(2d,d)\leq d$ just states the special case of $\diamfd_\Circuits(f,d)\leq f-d$ for $f=2d$.


Finally, we use a construction to prove $(2) \Rightarrow (1)$ and $(3) \Leftarrow (4)$; let $P$ be a $d$-dimensional polyhedron with $f$ facets.
Then, let \[Q = P \times [0,+\infty)^{f-2d}\]

Taking the product of $P$ with $[0,+\infty)$ increases the number of facets by $1$ and the dimension by $1$.
So the unbounded polyhedron $Q$ is in dimension $d + (f-2d) = f-d$, and has $f + (f-2d) = 2(f - d)$ facets.
By (4), $\Deltac(Q) \leq f-d$.

Also, by Lemma~\ref{circdiamprod}, $\Deltac(Q) = \Deltac(P) + \Deltac\left([0,+\infty)^{f-2d}\right)$.
But the polyhedron $[0,+\infty)$ has only one vertex and hence has circuit diameter $0$.
Thus, $\Deltac(Q) = \Deltac(P)$, which yields $\Deltac(P) \leq f-d$, which is (3).

In addition, any vertex-vertex circuit walk in $Q$ stays inside the factor $P \times \{\zero\}$, since any circuit step that increases one of the last $(f-2d)$ coordinates is unbounded.
By (2), given any two vertices $\veu \times \{\zero\}$ and $\vev \times \{\zero\}$ of $Q$ there is a non-revisiting circuit walk of length at most $f-d$ between them -- this is also a non-revisiting circuit walk between the corresponding vertices $\veu$ and $\vev$ in $P$, implying (1). 

\end{proof}

Note that if we restrict our attention to bounded polytopes then we can restate (1) to (4) in Theorem~\ref{maintheorem} as
\begin{enumerate}[(1b)]
\item Let $\veu,\vev$ be two vertices of a \Csimple\ polytope $P$.
	Then there is a non-revisiting circuit walk from $\veu$ to $\vev$.
\item Let $\veu,\vev$ be two vertices of a \Csimple\ $d$-dimensional polytope $P$ with $2d$ facets.
	Then there is a non-revisiting circuit walk from $\veu$ to $\vev$.
\item $\diamfd_\Circuits^b(f,d)\leq f-d$ for all $f\geq d$.
\item  $\diamfd_\Circuits^b(2d,d)\leq d$ for all $d$.
\end{enumerate}

The proof for $(2) \Leftarrow (1) \Rightarrow (3) \Rightarrow (4)$ carries over and gives us $(2b) \Leftarrow (1b) \Rightarrow (3b) \Rightarrow (4b)$.
However, the construction $Q = P \times [0,+\infty)^{f-2d}$ produces an unbounded polyhedron and so cannot be adapted to the bounded case -- so we do not have the corresponding implications $(2b) \Rightarrow (1b)$ nor $(3b) \Leftarrow (4b)$.

Observe that we are able to go from $(f,d)$ to $(2f-2d,f-d)$ in the above construction because each new factor increases both dimension and number of facets by one.
This is what the wedge construction in Definition~\ref{def:wedge} is used for in the combinatorial proof \cite{kw-67}; taking the wedge on $P$ over a facet also increases both dimension and number of facets by 1, while maintaining boundedness.

The classical proof in the combinatorial setting involves one additional statement on Dantzig figures (recall Definition~\ref{def:dantzig} and Conjecture~\ref{dantzig}):
\begin{quote}
(*) If $(P,\veu,\vev)$ is a $d$-dimensional Dantzig figure, then there is a circuit walk of length at most $d$ from $\veu$ to $\vev$.
\end{quote}
In the proof, one starts with a polyhedron $P = F_0$. Wedging repeatedly yields a sequence of polyhedra $F_1,\ldots,F_k$ with increasing dimension until a Dantzig figure is obtained. The edge walk in the Dantzig figure is non-revisiting, and it projects down to a non-revisiting edge walk in $P$.


In an attempt to adapt this idea for the circuit diameter, we introduced the concept of $k$-wedge-simplicity because it is necessary to ensure \Csimplicity\ of the final Dantzig figure $F_k$ produced; this gives a non-revisiting circuit walk in the Dantzig figure, which by (*) has length at most the dimension.
However this might not project down to a (non-revisiting) circuit walk in the intermediate wedge constructions, or the starting polyhedron itself. We first observed this problem due to hitting the strict interior of the upper base; recall Figure~\ref{img:wedge_circuit}. We have observed that the wedge operations in a sequence of wedges can be ordered arbitrarily, so hitting the strict interior of the upper base can be avoided. However, the same effect may also arise in the sides of the wedge, see Figure~\ref{fig:wedge_nonrevisit}. Due to this, the freedom in the ordering of the wedge operations does not help. For now, the relationship of Conjecture \ref{dantzig} to the other conjectures remains open.

\begin{figure}[htbp]
	\centering
 \begin{minipage}{.4\textwidth}
        \centering
	\includegraphics[width=0.9\linewidth]{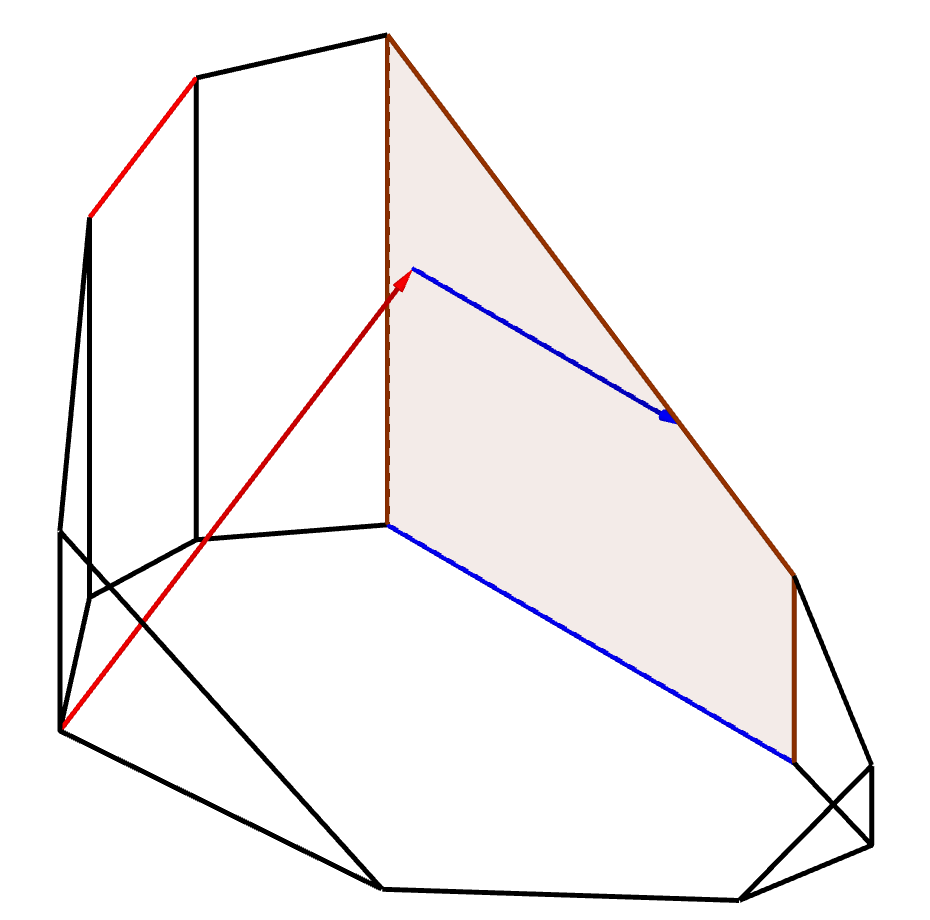}
    \end{minipage}%
    \quad
 \begin{minipage}{.4\textwidth}
        \centering
	\includegraphics[width=0.9\linewidth]{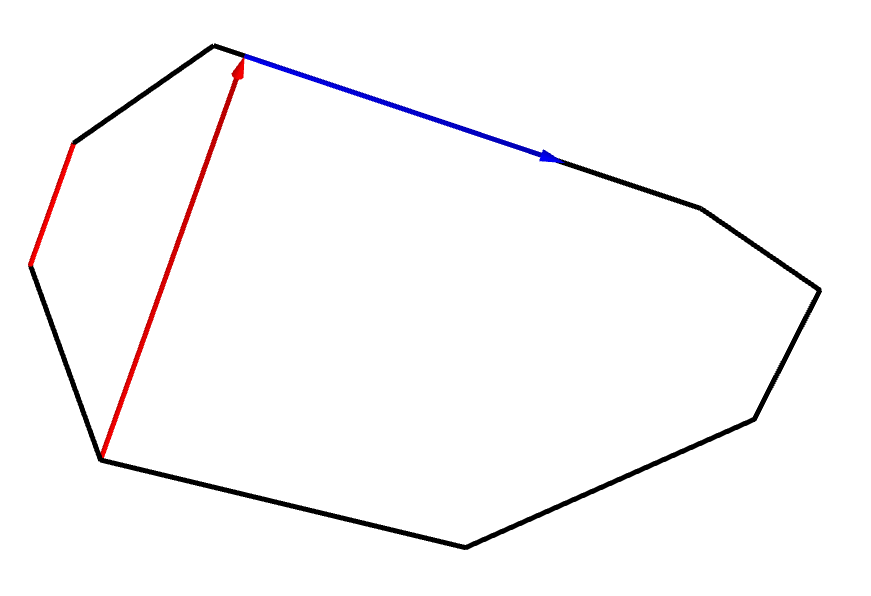}
    \end{minipage}%
    \caption[Non-revisiting circuit walks do not always transfer from a wedge to the original polyhedron]{The circuit walk in the wedge $P'$ on the left is non-revisiting. However it does not project to a circuit walk in $P$.}\label{fig:wedge_nonrevisit}
\end{figure}



\subsection{A connection of unbounded and bounded circuit diameters}

In the proof of Lemma \ref{lem:nonrev}, we added extra facets incident to a boundary point $\veu$ of $P$ to obtain a polyhedron $P'\subset P$ in which $\veu$ is a vertex.
The added facets were parallel to existing facets, such that $\mathcal{C}(P)=\mathcal{C}(P')$.
By performing a similar construction it is possible to transform an unbounded polyhedron $P$ to a bounded polytope $P'$ with $\mathcal{C}(P)=\mathcal{C}(P')$ and without cutting off any vertices.
We will do so by adding facets through $\veu$ that are parallel to facets through a particular vertex $\vev$.
In doing so, we obtain an intimate connection of the circuit diameters of bounded polytopes and unbounded polyhedra.

In the following, for a facet with outer normal $\vea_i$, we call a facet with outer normal $-\vea_i$ an {\em opposite facet}.
We begin by examining how many opposite facets have to be added to a single vertex of an unbounded polyhedron $P$ to obtain a bounded polytope $P'$.
Here we say a facet {\em blocks} an edge direction incident to vertex $\vev$ if the half-line starting at $\vev$ in edge direction intersects the facet.
Essentially, an edge direction only is unbounded if there is no blocking facet.

\begin{lemma}\label{lem:makebounded}
Let $P$ be a $d$-dimensional unbounded polyhedron with $f$ facets, and let $\veu,\vev$ be two vertices of $P$ that do not share a facet.
Then there is a bounded polyhedron $P'\subset P$, where $P'$ is constructed by adding at most $d-1$ facets incident to $u$ that are opposite to facets incident to $\vev$.
By this construction, $\veu,\vev\in P'$ are vertices and $\mathcal{C}(P)=\mathcal{C}(P')$. 
\end{lemma}

\begin{proof}
By construction, $P'$ still contains both $\veu$ and $\vev$, as all facets that are added are opposite facets of facets that are incident to $\vev$.
In particular, the facets are parallel to existing ones, so $\mathcal{C}(P)=\mathcal{C}(P')$.
Thus it suffices to prove that at most $d-1$ facets are necessary to create a bounded polytope.

Consider the recession cone of $P$.
In particular,  it is contained in the inner cone of $\vev$.
Thus it suffices to block the edge directions of the inner cone in $P'$ by the addition of extra facets.
A simple way to do so would be to add opposite facets incident to $\veu$ to {\em all} facets incident to $\vev$.
In fact, this would give a bounded box that contains $P'$.

However, not all of these facets are necessary.
Let $S$ be a simple cone coming from the selection of exactly $d$ linearly independent facets of the $d^*\geq d$ facets incident to $\vev$.
Then $S$ contains the inner cone and it suffices to block the corresponding $d$ edge directions of $S$. 

The graph of $P$ is connected and thus there is an edge incident to $\vev$ that leads to a neighboring vertex.
It is possible to choose the facets $F_1,\dots,F_d$ for $S$ as a superset of those $d-1$ facets that define such an edge.
This implies that we only have to block $d-1$ edge directions to validate the claim.

Let now $\vece$ be one of the edge directions of $S$ incident to $\vev$, where $\vece$ is defined by the intersection of $d-1$ facets.
There is exactly one facet $F_i$ with outer normal $\vea_i$ such that $\vea_i^T\vece<0$.
By adding the opposite facet incident to $\veu$ of $F_i$, the corresponding edge direction is blocked, i.e. it is not unbounded anymore.
By doing this for all $d-1$ edge directions that have to be blocked, one obtains a bounded polytope.
\end{proof}

Note that the number of facets that are necessary for the construction may be lower than $d-1$ if the recession cone is of lower dimension than $d$ or if multiple edges lead from $v$ to neighbors in the graph of $P$.
Lemma \ref{lem:makebounded} is our key ingredient to connect the maximal circuit diameters of unbounded $d$-dimensional polyhedra and the maximal circuit diameter of bounded $d$-dimensional polytopes.
This connection is based on validity of the non-revisiting conjecture for the bounded polytopes.

{
\renewcommand{\thetheorem}{\ref{thm:bounded_and_unbounded}}
\begin{theorem}
If all $\Circuits$-simple bounded $(f',d')$-polytopes with $f'\leq f+d-1$ and $d'\leq d$ satisfy the non-revisiting conjecture (Conjecture \ref{non revisiting}), then $\diamfd_\Circuits^u(f,d)\leq f-1$.
\end{theorem}
\addtocounter{theorem}{-1}
}


\begin{proof}
Let $P$ be an unbounded $d$-dimensional polyhedron with $f$ facets and let $\veu,\vev$ be two of its vertices.
We assume $\veu$ and $\vev$ do not share a facet; otherwise consider the minimal face that contains $\veu$ and $\vev$ in place of $P$.
Further, let $P'$ be constructed as in the proof of Lemma \ref{lem:makebounded}, where new facets are added to $u$.
Then $P'$ has at most $f+(d-1)$ facets.
Now perturb polytope $P'$ to obtain the $\Circuits$-simple polytope $P''$.
Let $\veu'$ be a vertex of $P''$ that corresponds to $\veu$ in the non-perturbed $P'$, and to which all (at most $d-1$) extra facets are incident.
(The existence of such a $\veu'$ in $P''$ can be guaranteed by first fixing a set of $d$ facets that contains all extra facets and relaxing all other facets slightly.
Any subsequent perturbation then keeps the single point of intersection $\veu'$ of these facets feasible, i.e. $\veu'$ is a vertex of $P''$.)
Let further $\vev'$ be a vertex of $P''$ that corresponds to $\vev$ in $P'$.

Now consider a non-revisiting circuit walk from $\veu'$ to $\vev'$ in $P''$.
As it is non-revisiting, none of the extra facets is the only blocking facet in any step -- in other words the step length is always bounded by one of the original facets.
This means that the circuit walk transfers to a circuit walk from $\veu$ to $\vev$ in $P$, with the same number of steps.
The circuit walk thus has length at most $f-1$, as none of the initial $d$ facets is revisited. 

\end{proof}
Note that the constructed walk for the unbounded polyhedron $P$ may be revisiting. An interesting special case arises for Dantzig figures. For this case, Theorem \ref{thm:bounded_and_unbounded} can be refined.

\begin{corollary}\label{cor:Dantzig}
If all $d$-dimensional bounded spindles $P(\veu,\vev)$ with $2d-1$ facets incident to $\veu$ and $d$ facets incident to $\vev$ have a non-revisiting circuit walk from $\veu$ to $\vev$, then all unbounded $d$-dimensional Dantzig figures $(P',\veu',\vev')$ have a non-revisiting circuit walk from $\veu$ to $\vev$.
\end{corollary}

Note the non-revisiting condition in the above corollary is needed to transfer a circuit walk in the bounded polytope 
to be a circuit walk in the unbounded polyhedron. In dimension $4$, all spindles have length at most $4$~\cite{sst-12}. 
Showing that such a walk can be realized in a non-revisiting manner gives the circuit $4$-step conjecture for 
bounded and unbounded polyhedra.  In the following section, we give two proofs of the circuit $4$-step conjecture:
first by a careful analysis of the Klee-Walkup polyhedron $U_4$ and second via Corollary~\ref{cor:Dantzig}
by showing that $4$-spindle walks can in fact be made non-revisiting.

\section{The Circuit 4-step Conjecture}\label{sec:4step}

\subsection{Proof via the Klee-Walkup polyhedron}

The first unbounded counterexample to the Hirsch conjecture was given by Klee and Walkup in \cite{kw-67}, where they constructed a 4-dimensional polyhedron with 8 facets and combinatorial diameter 5. In the extended abstract \cite{sy-15b}, Stephen and Yusun prove that this polyhedron satisfies the Hirsch bound in the circuit diameter setting. We detail the proof here, and then consider the more general 4-step conjecture afterwards.

Denote by $U_4$ the polyhedron defined by the system of linear inequalities $\{\vex \in \R^4 : A\vex \geq \veb\}$, $$\text{  where  } A = \begin{pmatrix} -6 & -3 & \ \ 0 & \ \ 1 \cr
-3 & -6 & 1 & 0 \cr
-35 & -45 & 6 & 3 \cr
-45 & -35 & 3 & 6 \cr
1 & 0 & 0 & 0 \cr
0 & 1 & 0 & 0 \cr
0 & 0 & 1 & 0 \cr
0 & 0 & 0 & 1  \end{pmatrix}
\text{ and }
\veb = \begin{pmatrix}
-1 \cr -1 \cr -8 \cr -8 \cr 0 \cr 0 \cr 0 \cr 0
\end{pmatrix}. $$

Its graph is shown in Figure~\ref{U4graph}: here the vertices are indexed by the four facets containing each one, while the points labelled with \texttt{R}'s represent extreme rays. It is clear from the graph that vertices \texttt{V5678} and \texttt{V1234} are at graph distance five apart. 

\begin{figure}[h]
\begin{center}
\includegraphics[width=0.65\textwidth]{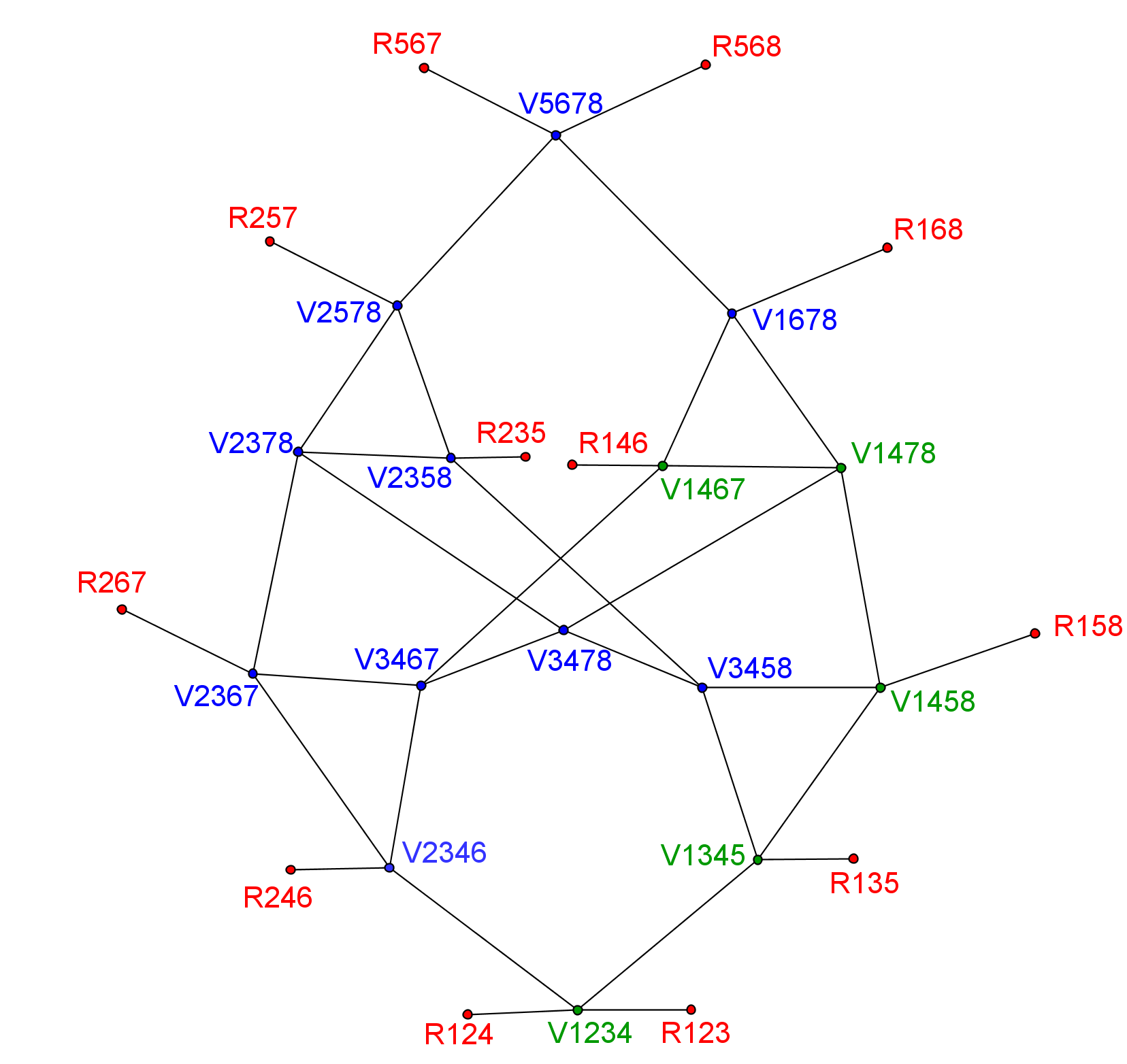}
\caption{The graph of $U_4$.}\label{U4graph}
\end{center}
\end{figure}

The result we need is the following:

\begin{theorem}\label{U4thm}
The circuit diameter of the Klee-Walkup polyhedron $U_4$ is at most 4, independent of realization.
\end{theorem}
\begin{proof}
First we demonstrate the existence of a circuit walk of length 4 from \texttt{V5678} to \texttt{V1234}. Observe that we can take two edge steps as follows: \texttt{V5678} $\rightarrow$ \texttt{V1678} $\rightarrow$ \texttt{V1478}. Vertices \texttt{V1478} and \texttt{V1234} are both contained in the 2-face determined by facets \texttt{1} and \texttt{4}, so we can complete the walk on this face. Note that this 2-face is an unbounded polyhedron on six facets. Figure~\ref{U4face14} is a topological illustration of this face, showing the order of the vertices and rays.

\begin{figure}[htbp]
\begin{center}
\includegraphics[width=0.35\textwidth]{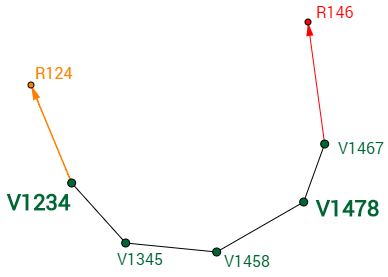}
\caption{The 2-face determined by facets \texttt{1} and \texttt{4}.} \label{U4face14}
\end{center}
\end{figure}

Now consider a vector $\veg$ corresponding to the edge direction from \texttt{V1458} to \texttt{V1345} -- this is the blue vector in Figure~\ref{U4face14_feas}. Note that this is always a circuit direction in any realization of $U_4$ since it corresponds to an actual edge of the polyhedron.

To see that $\veg$ is a feasible direction at \texttt{V1478}, consider vector $\veh$ in the edge direction from \texttt{V1478} to \texttt{V1458}, and vector $\ver$ in the direction of ray \texttt{R124}. Observe that $\veg$ and $-\veh$ are the two incident edge directions at \texttt{V1458}, and so $\ver$ must be a strict conic combination of $\veg$ and $-\veh$, i.e. $\ver = \alpha_1 (\veg) + \alpha_2 (-\veh)$ for $\alpha_1, \alpha_2 > 0$. By rearranging terms we see that $\veg$ is a strict conic combination of $\veh$ and $\ver$: $\veg = (\alpha_2/\alpha_1)\veh + (1/\alpha_1)\ver$, with $\alpha_2/\alpha_1, 1/\alpha_1 > 0$. Feasibility of $\ver$ and $\veh$ at \texttt{V1478} implies that $\veg$ is a feasible direction at \texttt{V1478}.

\begin{figure}[htbp]
\begin{center}
\includegraphics[width=0.4\textwidth]{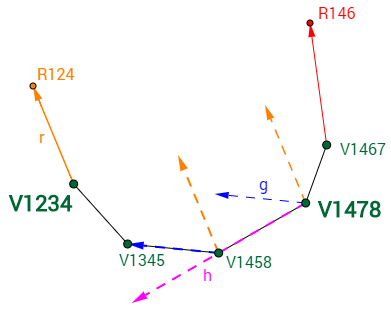}
\caption{Feasibility of the circuit direction $\veg$.} \label{U4face14_feas}
\end{center}
\end{figure}

Now starting at \texttt{V1478} traverse $\veg$ as far as feasibility allows. This direction is bounded since we exit the polyhedron when taking $\veg$ from \texttt{V1458}. We will eventually exit the 2-face at a point along the boundary, and at one of the following positions:

\begin{itemize}
\item exactly at \texttt{V1234},
\item on the edge connecting \texttt{V1234} and \texttt{V1345}, or
\item on the ray \texttt{R124} emanating from \texttt{V1234}.
\end{itemize}

Hitting exactly \texttt{V1234} gives a circuit walk of length 3 from \texttt{V5678}, while the other two cases give circuit walks of length 4 since we only need one step to \texttt{V1234}. These two situations are illustrated in Figure~\ref{U4face14b}.

\begin{figure}[htbp]
\begin{center}
\includegraphics[width=0.75\textwidth]{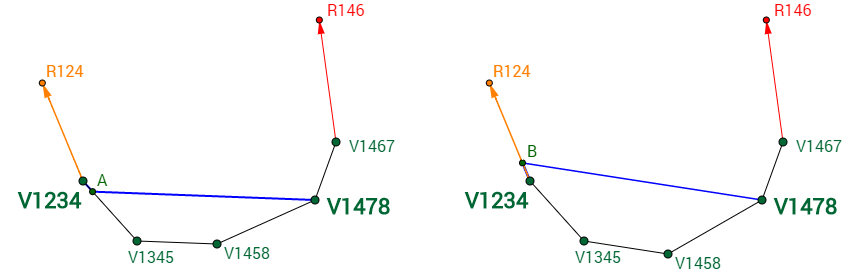}
\caption{Getting from \texttt{V1478} to \texttt{V1234} in at most 2 steps.} \label{U4face14b}
\end{center}
\end{figure}

The argument is the same for the reverse direction (\texttt{V1234} to \texttt{V5678}). We can construct a similar walk by first traversing edges \texttt{V1234} $\rightarrow$ \texttt{V2346} $\rightarrow$ \texttt{V3467}, and then taking a maximal step in the circuit direction arising from the edge connecting \texttt{V1467} and \texttt{V1678}. Here we stay in the 2-face determined by facets \texttt{6} and \texttt{7}. We can then arrive at \texttt{V5678} in at most two steps from \texttt{V3467}. 
\end{proof}

We refer to a polyhedron with $f$ facets in dimension $d$ as an $(f,d)$-polyhedron.
One consequence of Theorem~\ref{U4thm} is the general circuit 4-step conjecture, however we will need the following fact to prove it:

\begin{lemma}\label{U4unique}
Up to isomorphism, $U_4$ is the only non-Hirsch simple (8,4)-polyhedron. 
\end{lemma}
\begin{proof}
The simple bounded (8,4)-polytopes are enumerated in Gr\"unbaum and Sreedharan~\cite{gs-67}.
A simple unbounded (8,4)-polyhedron can be truncated with an additional facet that cuts off
the vertex at infinity, this produces a simplicial (9,4)-polytope.  Observe that the diameter
of this polytope will remain at least 5, as any route between \texttt{V1234} and \texttt{V5678}
through the new facet will have to add the new ninth facet along with facets 5, 6, 7 and 8.
Klee and Kleinschmidt in \cite{kk-87} mention that there is a unique simple polytope with 
$d = 4$, $f = 9$ and diameter 5, following directly from the complete enumeration of all polytopal simplicial
3-spheres, completed by Altshuler et al.~\cite{abs-80}.  Thus this must be exactly that polytope,
which we denote by $Q_4$. 
The result follows, as any non-Hirsch 4-polyhedron with 8 facets comes from $Q_4$ by projecting to infinity the ninth facet that does not contain either of the two vertices at distance 5.
\end{proof}


{
\renewcommand{\thetheorem}{\ref{thm:4step}}
\begin{theorem}[Circuit 4-step]
$\diamfd_\Circuits(8,4) = 4.$
\end{theorem}
\addtocounter{theorem}{-1}
}

\begin{proof}
By Lemma~\ref{Csimple}, it suffices to consider \Csimple\ polyhedra -- let $P$ be a 4-dimensional polyhedron with 8 facets. If $P$ is bounded then it has combinatorial diameter 4 \cite{sst-12}, so suppose $P$ is unbounded. By Lemma~\ref{U4unique}, $P$ is combinatorially equivalent to $U_4$, and by Theorem~\ref{U4thm} it has circuit diameter 4.
\end{proof}

We remark that in fact the circuit walks produced between the two vertices at distance 5 are non-revisting, so in fact this argument shows
that there are non-revisiting circuit walks between any pair of vertices of an (8,4)-polyhedron.


\subsection{Proof via facial paths in 4-prismatoids}

Here we present a second proof of Theorem~\ref{thm:4step}. Recall that a {\em spindle} is a polytope with two distinguished vertices $\vex$ and $\vey$ such that each facet is incident to exactly one
of $\vex$ and $\vey$. Polar to this, a {\em prismatoid} is a polytope with two distinguished facets $Q^+$ and $Q^-$ (called its {\em bases}) that together contain all the vertices of the polytope. The
{\em length} of a spindle is the graph distance between the two special vertices, while the {\em width} of a prismatoid is the dual graph distance between the two bases.
These constructions were essential in finding counterexamples to the combinatorial Hirsch conjecture \cite{s-11}.
In \cite{sst-12} Santos et al.~prove that 4-dimensional prismatoids have width at most 4. We strengthen this
result here to get a non-revisiting path and use it to show $\diamfd_\Circuits(8,4) = 4$.

\begin{lemma}\label{4prismatoid}
In a 4-prismatoid with parallel faces $Q^+$ and $Q^-$, there exists a facial path from $Q^+$ to $Q^-$ such that at each step at least one new vertex of $Q^-$ is encountered.
\end{lemma}
\begin{proof}
Suppose a 4-prismatoid $Q$ is given, with bases $Q^+$ and $Q^-$.
If $Q$ has width 2 then there is a facet of $Q$ that is adjacent to both bases.
The claim then follows as the number of vertices of $Q^-$ incident to each facet in the facial path is strictly increasing.

If $Q$ has width 3, suppose the facial path of length 3 is $Q^+\rightarrow F \rightarrow G \rightarrow Q^-$.
Then $F$ must have at most 2 vertices from $Q^-$ -- any more and it would itself be adjacent to $Q^-$, and there would be a shorter path between the bases.
Also, $G$ must have at least 3 vertices in common with $Q^-$ to be adjacent to it.
Hence the number of vertices of $Q^-$ incident to each facet in this facial path is also strictly increasing.

Suppose now that $Q$ has width 4.
Santos et al. prove in \cite{sst-12} that there is a facial path of length 4 between the bases, say $Q^- \rightarrow F_1 \rightarrow F_2 \rightarrow F_3 \rightarrow Q^+$, such that $F_2$ is a tetrahedron with $|V(Q^-) \cap V(F_2)| = 2$ and $|V(Q^+) \cap V(F_2)| = 2$.
That is, $F_2$ has two vertices on $Q^-$ and two vertices on $Q^+$.
Otherwise, if $F_2$ were incident to more than two vertices of say, $Q^+$, then it would be adjacent to $Q^+$ and we would have a shorter facial path between the bases, contradicting the assumption that $Q$ has width 4.
Denote by $\vex$ and $\vey$ ($\vez$ and $\vew$) the two vertices of $Q^+$ ($Q^-$) on $F_2$.

Let us now consider each step of the path.
The first step $Q^+ \rightarrow F_1$ and the last step $F_3 \rightarrow Q^-$ clearly satisfy the condition we require.
Moreover, going from $F_2$ to $F_3$, the number of vertices on $Q^-$ increases from 2 to at least 3 -- a strict increase as well.

As for the step from $F_1$ to $F_2$, the crucial observation is that the triangle of $F_2$ that is incident to $F_1$ contains two vertices of $Q^+$ ($\vex$ and $\vey$) and one of $Q^-$ (assume without loss of generality that it is $\vez$).
This means that $F_1$ cannot contain $\vew$ as well, or else $F_1$ would contain $F_2$ entirely.
Therefore $\vew$ is the new vertex of $Q^-$ seen when moving from $F_1$ to $F_2$.

\end{proof}

Note that although similar, the notion of seeing a new vertex of $Q^-$ in each step is not the same as the interpretation of non-revisiting paths discussed in Section~\ref{sec:nonrevisiting}. Here we do not require that the new vertex at each step be new with respect to the entire path, just that it is a new vertex of $Q^-$ at that step. 

Figure~\ref{fig:prismatoid} illustrates a facial path from the outer facet to the inner facet in a 4-prismatoid, using a partial Schlegel diagram; observe that this is a \emph{revisiting} path since the vertex $F$ is left at the second step and then seen again in the last step.
However this path still satisfies the condition we need, that at least one new vertex of $Q^-$ is seen \emph{at each step}: $E$ and $F$ for step 1, $D$ for step 2, $A$ for step 3, and $B$, $C$, and $F$ for step 4 (we still list $F$ here because it is a vertex of $Q^-$ but not of $F_3$, although it has already appeared in the walk before).

\begin{figure}[htbp]
\begin{center}
	\includegraphics[width=0.4\textwidth]{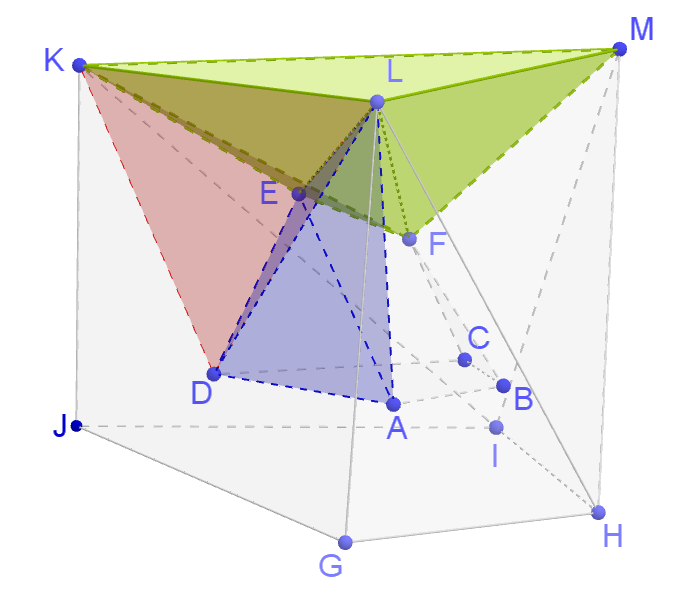}
	\caption[A revisiting facial path in a 4-prismatoid]{A revisiting facial path in a 4-prismatoid. The facet $F$ is left after the first step but is entered again in the last step.}\label{fig:prismatoid}
	\end{center}
\end{figure}

Now we use the polar result for spindles, which in turn implies the next result:

\begin{corollary}\label{4spindle}
Let $P(\veu,\vev)=P(\veu)\cap P(\vev)\subset \R^4$ be a bounded spindle coming from the intersection of two cones $P(\veu), P(\vev)$ at $\veu$, respectively $\vev$.
Then there is an edge walk from $\veu$ to $\vev$ such that in each step, a new facet of $P(\vev)$ becomes active.
\end{corollary}

\begin{lemma}\label{lem:4stepsunbounded}
Let $P(\veu,\vev)=P(\veu)\cap P(\vev)\subset \R^4$ be an unbounded spindle coming from the intersection of two cones $P(\veu), P(\vev)$ at $u$, respectively $\vev$.
Let further $P(\vev)$ be simple.
Then there is a circuit walk of length at most $4$ from $\veu$ to $\vev$. 
\end{lemma}

\begin{proof}
Let $\vea_1,\dots,\vea_4$ be the outer normals of facets $F_1,\dots,F_4$ incident to $\vev$.
Let further $Q_i=\{\vex\in \R^d: (-\vea_i)^T\vex\leq (-\vea_i)^T\veu\}$, informally an opposite halfspace of the one created by $F_i$, but now moved to be incident to $\veu_i$.
Set $P'(\veu)=P(\veu)\cap Q$ with $Q=\bigcap\limits_{i=1}^4 Q_i$ and $P'(\veu,v)=P'(\veu)\cap P(\vev)$.
Clearly, $P'(\veu,\vev)\subset P(\veu,\vev)$ is a bounded spindle with simple cone $P(\vev)$ so that we may apply Corollary~\ref{4spindle}.

Thus there is a circuit walk from $\veu$ to $\vev$ in $P'(\veu,\vev)$ of length at most $4$ such that in each step at least one of the facets of $P(\vev)$ becomes active.
This means that the `extra' facets introduced as $Q$ never are the only facets to bound the step length.
Combining this fact with $P'(\veu,\vev)\subset P(\veu,\vev)$, we see that  the circuit walk from $\veu$ to $\vev$ in $P'(\veu,\vev)$ of length at most $f$ is a circuit walk in $P(\veu,\vev)$, as well. This proves the claim.
\end{proof}

The bounded case $\Deltac(8,4) \leq 4$ already follows from the combinatorial diameter bound $\Deltae(8,4) \leq 4$; hence Lemma~\ref{lem:4stepsunbounded} takes care of the only other possible bad case -- when the $(8,4)$-polyhedron is unbounded.
So, given an unbounded $(8,4)$-polyhedron, we take two of its vertices $\veu$ and $\vev$, which we can assume to have no facets in common (otherwise we reduce to the $3$- or fewer-dimensional case.)
Then an application of Lemma~\ref{lem:4stepsunbounded} gives Theorem~\ref{thm:4step}.



\section{Discussion}\label{sec:discussion}

The (combinatorial) polynomial Hirsch conjecture remains a fundamental question.
We consider a diameter question where we make the natural relaxation of edge walks to circuit walks. 
In this setting, the Hirsch bound of $f-d$ is again a possibility.  We recover part of the Klee-Walkup equivalences
for the Hirsch bound holding in this setting. In particular, we show that the non-revisiting and $d$-step statements for circuit walks imply the circuit analogue of the Hirsch conjecture. 
We find an obstacle to using the wedge construction for circuits. Due to this, we prove some implications through the construction of a suitable unbounded polyhedron. Two
questions remain open in this context: Does the circuit diameter conjecture (Conjecture \ref{original}) imply the non-revisiting conjecture (Conjecture \ref{non revisiting})? And is it possible to
reduce to the $(2d,d)$ case without the construction of an unbounded polyhedron? The latter would allow such reductions for the bounded versions of the circuit non-revisiting and circuit Hirsch
conjectures.

Further, we show that the unbounded $4$-step conjecture holds for the circuit diameter, which fails for the combinatorial version. Presently we do not see a clear path to resolve the general $d$-step conjecture, though the circuit 5-step conjecture may be approachable via either of the approaches that worked for 4-step. An enumerative approach would argue that any simple (10,5)-polyhedra are obtained fairly
directly from $U_4$, allowing us to leverage the construction in this paper.
While this seems likely, see for example the discussion in Kim and Santos~\cite{ks-10},
it may require substantial enumeration to prove.
The discussion in Firsching's thesis~\cite{f-15} outlines the state-of-the-art for enumeration.
Alternatively, it could be resolved by proving that $5$-dimensional spindles with $9$ facets on 
one side and $5$ on the other side satisfy the non-revisiting conjecture.

It is interesting to consider \emph{circuit Hirsch-sharp} polyhedra, i.e.~polyhedra
that meet the $f-d$ bound.  The combinatorial Hirsch-sharp polytopes were intensively 
studied~\cite{fh-99,hk-98,hk-98b} before the demise of the bounded Hirsch conjecture.
These include {\em trivial} Hirsch-sharp polytopes (where $f \leq 2d$) like the $d$-cube and the $d$-simplex, and {\em non-trivial} Hirsch-sharp polytopes (where $f > 2d$), which include the polytope $Q_4$, and others obtained by performing operations on $Q_4$ (see \cite{ks-10}). In the circuit setting, the $d$-simplex remains Hirsch-sharp
independent of realization since $f-d=1$.  A regular $d$-cube is also Hirsch-sharp, but it is not obvious whether this remains true for non-regular realizations. Similarly, it is open whether there is a collapsing realization of $U_4$, i.e. one with a diameter of less than $4$. 

While almost all realizations of $U_4$ are Hirsch-sharp, it is difficult to determine if there exists a Hirsch-sharp realization of $Q_4$.
Surprisingly, if there is a Hirsch-sharp realization of $Q_4$, it will be sharp in only one direction: if $\veu$ and $\vev$ are the vertices at combinatorial distance 5, it is always possible to find a circuit walk of length $4$ from \emph{either} $\veu$ to $\vev$ or from $\vev$ to $\veu$~\cite{y-17}.
However, this does not transfer to the opposite direction -- note that reversing a (maximal) circuit walk does not give a circuit walk. 





\section*{Acknowledgments}
This research was partially supported by an NSERC Discovery Grant for T.~Stephen, and an NSERC Postgraduate Scholarship-D for T.~Yusun. All illustrations were produced using the GeoGebra
software.\footnote{http://www.geogebra.org, International GeoGebra Institute.}
We thank J.~de Loera and E.~Finhold for comments and encouragement.

\bibliographystyle{amsalpha}
\bibliography{literature}

\end{document}